\documentclass[11pt, letterpaper]{amsart}
\usepackage{amsmath,amssymb,hyperref}

\usepackage{tikz}
\usetikzlibrary{cd}

\usepackage[bb=ams,scr=euler]{mathalpha}

\usepackage{enumitem}

\usepackage{thmtools}
\declaretheoremstyle[headfont=\normalsize\normalfont\bfseries,notefont=\mdseries,
notebraces={(}{)},bodyfont=\normalfont\itshape,postheadspace=0.5em]{italstyle}

\declaretheorem[style=italstyle,name=Theorem]{theorem}
\declaretheorem[style=italstyle,name=Lemma,sibling=theorem]{lemma}
\declaretheorem[style=italstyle,name=Proposition,sibling=theorem]{proposition}
\declaretheorem[style=italstyle,name=Question]{question}

\makeatletter
\newcommand{\abs}[1]{\left|#1\right|}
\newcommand{\bd}{\partial}
\newcommand{\C}{\mathbb{C}}
\renewcommand{\d}{\mathrm{d}}
\newcommand{\id}{\mathrm{id}}
\newcommand{\norm}[1]{\left\lVert#1\right\rVert}
\newcommand{\R}{\mathbb{R}}
\newcommand{\set}[1]{\left\{#1\right\}}
\newcommand{\Z}{\mathbb{Z}}

\renewcommand\section{\@startsection{section}{1}{0pt}{-3.5ex \@plus -1ex \@minus -.2ex}{2.3ex \@plus.2ex}{\centering\itshape}}
\renewcommand{\subsection}{\@startsection{subsection}{2}\z@{.5\linespacing\@plus.7\linespacing}{-.5em}{\normalfont\itshape}}
\makeatother

\author{Dylan Cant}
\email{dylan@dylancant.ca}
\address{Institut de mathématique d'Orsay, Université Paris-Saclay, Bâtiment 307, rue Michel Magat, F-91405 Orsay Cedex, France}

\author{Jun Zhang}
\email{jzhang4518@ustc.edu.cn}
\address{The Institute of Geometry and Physics, University of Science and Technology of China, 96 Jinzhai Road, Hefei, Anhui, 230026, China}

\date{\today}
\begin{document}
\title{On the spectral capacity of submanifolds}
\begin{abstract}
  The infimum of the spectral capacities of neighbourhoods of a nowhere coisotropic submanifold is shown to be zero. In contrast, neighbourhoods of a closed Lagrangian submanifold, and of certain contact-type hypersurfaces, are shown to have uniformly positive spectral capacity. Along the way we prove a quantitative Lagrangian control estimate relating spectral invariants, boundary depth, and the minimal area of holomorphic disks. The Lagrangian control also provides novel obstructions to certain Lagrangian embeddings into a symplectic ball.
\end{abstract}
\maketitle

\section{Introduction}
\label{sec:introduction}

Let $(M^{2n},\omega)$ be a semiconvex symplectic manifold. Here \emph{semiconvexity} is a geometric assumption which states that the non-compact end of $M$, if non-empty, is modelled on $S_{+}Y\times T$ where $S_{+}Y$ is the positive half of the symplectization of a compact contact manifold and $T$ is a symplectically aspherical symplectic manifold (e.g., $T$ could be a point). For example, all compact symplectic manifolds are semiconvex, the open symplectic manifolds $\C^{n},T^{*}L$ are semiconvex, and if $M$ is semiconvex then so is the Cartesian product $M\times T^{2}$.

We will also require the following notion of semipositivity. Following, e.g., \cite{hofer-salamon-95,seidel-GAFA-1997,mcduff-salamon-book-2012}, we say that $(M,\omega)$ is \emph{semipositive} if any smooth sphere $u:S^{2}\to M$ satisfies:
\begin{equation*}
  \omega(u)>0\text{ and }c_{1}(u)\ge 3-n\implies c_{1}(u)\ge 0.
\end{equation*}
If $3-n$ is replaced by $2-n$, then we say $M$ is \emph{strongly semipositive}. It will be important for us to observe that, if $M$ is semiconvex and strongly semipositive, then $M\times T^{2}$ is semiconvex and semipositive.

If $M$ is semiconvex and semipositive, \cite{schwarz-pacific-j-math-2000,oh-2005-duke,frauenfelder-schlenk-IJM-2007,usher-compositio-2008} explain how to associate a \emph{spectral invariant} $c(H_{t},[M])$ to any compactly supported time-dependent Hamiltonian\footnote{We abuse notation and let the symbol $H_{t}$ represent a time-dependent Hamiltonian function $M\times \R/\Z\to \R$.} function $H_{t}$ on $M$ via a homological min-max process applied to a certain class in the \emph{Floer homology} of $H_{t}$. Similar invariants appear in \cite{viterbo-mathann-1992} using instead the theory of generating functions.

Given any open subset $U\subset M$ we define the \emph{spectral capacity} by:
\begin{equation*}
  c(U):=\sup\set{c([M],H_{t}):H_{t}\text{ has compact support in }U}.
\end{equation*}
Finally, for any other set $N$, we extend the spectral capacity by outer regularity, i.e.,  the infimum is over open sets $U$.
\begin{equation*}
  c(N)=\inf\set{c(U\subset M):U\text{ is an open neighbourhood of }N}
\end{equation*}
Such a quantity has appeared before, and is called the \emph{homological capacity} in \cite[\S3.3.4]{ginzburg-duke-2007} and the \emph{spectral width} in \cite{polterovich-rosen-book-2014}.

Our first theorem proves the spectral capacity vanishes for a large class of submanifolds, including all compact symplectic submanifolds of positive codimension:
\begin{theorem}\label{theorem:1}
  If $M$ is semiconvex and strongly semipositive and $S\subset M$ is a compact nowhere coisotropic submanifold then $c(S)=0$.
\end{theorem}

The proof is given in \S\ref{sec:proof-theorem-1}. The main idea appeals to the fact that $S$ is \emph{stably infinitesimally displaceable}, as proven by \cite{gurel-CCM-2008} and \cite{laudenbach-sikorav-IMRN-1994}. The strategy is to relate spectral invariants in $M$ with spectral invariants in $M\times T^{2}$.

It is well-known that, if $S$ can be displaced by Hamiltonian isotopies with arbitrarily small Hofer length then the energy-capacity inequality implies Theorem \ref{theorem:1}; see, e.g., \cite[\S5.5]{hofer-zehnder-book-1994}, and \cite{schwarz-pacific-j-math-2000,oh-2005-duke,frauenfelder-ginzburg-schlenk,ginzburg-2005-weinstein,usher-CCM-2010} for details on the energy-capacity inequality. Our argument follows a similar idea, and the key step is to generalize the energy-capacity inequality to work for \emph{stable displacement energy}; in Theorem \ref{theorem:stable-displacement} below, we prove that the stable displacement energy of a set bounds its spectral capacity from above.

\subsection{Note on relative capacities}
\label{sec:notational-conventions}

The rest of the introduction is concerned with other results about the spectral capacity. Many of the results involve introducing other capacities. In total, we will introduce three Floer theoretic capacities $c,\gamma,\beta$, and a Lagrangian capacity $\ell$.

These capacities are all relative capacities, in that they are always associated to pairs $(N, M)$, where $M$ is a symplectic manifold and $N\subset M$ is an arbitrary subset. All of the capacities we consider are first defined for open subsets and then extended to all sets by \emph{outer regularity}:
\begin{equation*}
  c(N, M)=\inf\set{c(U, M):U\text{ is an open neighborhood of }N}.
\end{equation*}
Each capacity is a functor valued in the category $(\R,\le)$ whose objects are the real numbers with a morphism $a\to b$ if and only if $a\le b$. Here a \emph{morphism of pairs} is a symplectomorphism $\varphi:M_{1}\to M_{2}$ so that $\varphi(N_{1})\subset N_{2}$. In particular, if $\varphi$ is a symplectomorphism, then:
\begin{equation*}
  c(N_{1}, M_{1})=c(\varphi(N_{1}), \varphi(M_{1})).
\end{equation*}
As the ambient space is typically clear from the context, we will use the abbreviation $c(N)=c(N, M)$.

\subsection{The spectral diameter versus the capacity}
\label{sec:spectr-diam-vers}

An important quantity related to the spectral capacity is the \emph{spectral norm}:
\begin{equation*}
  \gamma(\varphi_{t}):=c([M],H_{t})+c([M],\bar{H}_{t}),
\end{equation*}
where $H_{t}$ generates $\varphi_{t}$ and $\bar{H}_{t}=-H_{t}\circ \varphi_{t}$. One advantage of this quantity is that it is independent of the choice of the Hamiltonian function $H_{t}$, and depends only on $\varphi_{t}$. Analogously to the definition of $c(U)$, one can define:
\begin{equation}
  \gamma(U):=\sup\set{\gamma(\varphi_{t}):\varphi_{t}\text{ is supported in }U},
\end{equation}
extended to all subsets $S\subset M$ by outer regularity. Clearly $\gamma(S)\le 2c(S)$; thus Theorem \ref{theorem:1} implies $\gamma(S)=0$ holds whenever $S$ is a compact nowhere coisotropic submanifold.

\subsection{Lagrangian control for the spectral capacity}
\label{sec:contr-with-lagr}

Our next result bounds the spectral capacity of a compact Lagrangian $L$ in terms of the areas of holomorphic disks with boundary on $L$. Following, e.g., \cite{chekanov-duke-1998}, define: $$\hbar(L):=\sup\set{\hbar(L,J):J\in \mathscr{J}}$$ where $\hbar(L,J)$ is the minimal area of a $J$-holomorphic disk in $M$ with boundary on $L$, or $J$-holomorphic sphere in $M$. Here $\mathscr{J}$ is the set of \emph{admissible} almost complex structures on $M$, namely those which are $\omega$-tame and invariant under the Liouville flow in the noncompact end $S_{+}Y\times T$ (the Liouville flow acts only on the first factor). Let us note that if $J$ is admissible and $J'$ is $\omega$-tame and differs from $J$ only on a compact set, then $J'$ is also admissible. It follows that $\hbar(L)=\hbar(L')$ whenever $L,L'$ differ by an exact isotopy.

We will show:
\begin{theorem}\label{theorem:2}
  If $M$ is semiconvex and semipositive, and if $L\subset M$ is a compact Lagrangian submanifold, then $\hbar(L)\le c(L)$.
\end{theorem}

Briefly, the idea is that for any Hamiltonian system one can try to define a version of the open-closed map on the Floer homology by counting half-infinite Floer cylinders with boundary on $L$, as shown in Figure \ref{fig:lagrangian-estimator}.

By counting the rigid such cylinders passing through a fixed point in $L$, we will conclude the following Lagrangian control\footnote{See \cite{polterovich-IMRN-1998,polterovich-book-2001,leclercq-zapolsky-JTA-2018,polterovich-shelukhin-compositio-2023} for variations on the Lagrangian control property.} property:
\begin{lemma}\label{lemma:3}
  If $H_{t}$ is a compactly supported Hamiltonian function on $M$ and:
  \begin{equation*}
    \beta(H_{t})+\int_{0}^{1}\max(H_{t})-\min(H_{t}|_{L})dt<\hbar(L),
  \end{equation*}
  then we have:
  \begin{equation*}
    c([M],H_{t})\ge \int_{0}^{1}\min(H_{t}|_{L})dt;
  \end{equation*}
  here $\beta(H_{t})$ is the boundary depth of \cite{usher-israel-j-math-2011}.
\end{lemma}
The proof of Lemma \ref{lemma:3} is deferred to \S\ref{sec:proof-lemma-3}.

It is important to recall from \cite{usher-israel-j-math-2011,oh-JKMS-2009} that:
\begin{equation}\label{eq:boundary-depth-bound}
  \beta(H_{t})\le \int_{0}^{1}\max(H_{t})-\min(H_{t})dt.
\end{equation}
Note that \eqref{eq:boundary-depth-bound} is generalized to $\beta(H_{t})\le \gamma(H_{t})$ in \cite{kislev-shelukhin-GT-2021,feng-zhang-arXiv-2024}.

\begin{proof}[Proof of Theorem \ref{theorem:2}]
  for any $A<\hbar(L)$, and any neighborhood $U$ of $L$, one can find a function $H_{t}$ with:
  \begin{enumerate}[label=(\alph*)]
  \item $\max(H_{t})=\min(H_{t}|_{L})=A$,
  \item $\min(H_{t})=0$,
  \item $H_{t}$ is compactly supported in $U$,
  \end{enumerate}
  and so that the hypotheses of Lemma \ref{lemma:3} are satisfied, using \eqref{eq:boundary-depth-bound}. One therefore concludes $c(U)\ge \hbar(L)$, and taking the infimum over neighborhoods $U$ yields Theorem \ref{theorem:2}.
\end{proof}

\begin{figure}[h]
  \centering
  \begin{tikzpicture}[yscale=1]
    \draw (0,0) circle (0.3 and .6) +(0,-.6) coordinate (A) +(0,.6) coordinate(B) +(-0.3,0) node[left] {$\gamma$};
    \draw (8,0) circle (0.3 and .6) +(0,-.6) coordinate (C) +(0,.6) coordinate(D) +(0.3,0) node[right] {$L$};
    \draw (A)--(C)node[fill,circle,inner sep=1pt]{} (B)--(D);
    \path (A)--coordinate(E)(D);
    \node at (E) {$\bd_{s}u+J(u)(\bd_t u-X_t(u))=0$};
  \end{tikzpicture}
  \caption{A half-infinite Floer cylinder asymptotic to $\gamma$ with Lagrangian boundary conditions is used to estimate the spectral invariant.}
  \label{fig:lagrangian-estimator}
\end{figure}
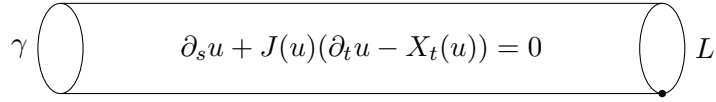

\subsubsection{Lagrangian and boundary depth capacities}
\label{sec:lagr-bound-depth}

We digress for a moment on two additional capacities. First, we introduce the \emph{Lagrangian-$\hbar$ capacity} of an open set $U\subset M$ to be:
\begin{equation*}
  \ell(U):=\sup \set{\hbar(L):L\subset U\text{ is a closed Lagrangian}}.
\end{equation*}
Second, we introduce the \emph{boundary depth capacity} of $U\subset M$ to be:
\begin{equation*}
  \beta(U):=\sup\set{\beta(H_{t}):H_{t}\text{ is supported in $U$}}.
\end{equation*}
Both capacities are extended to all subsets $N\subset M$ by outer regularity.

Lemma \ref{lemma:3} implies:
\begin{theorem}
  Suppose that $c(N)$ is finite. Then $\ell(N)\le \beta(N).$
\end{theorem}
\begin{proof}
  Pick any neighborhood $N\subset U$ small enough so that $c(U)$ is finite, and pick a Lagrangian $L\subset U$ and complex structure $J$. As above, choose a function $H_{t}$ so that $\max(H_{t})=\min(H_{t}|_{L})=A$.

  Provided $\beta(H_{t})$ is smaller than $\hbar(L)$, Lemma \ref{lemma:3} implies that $c(H_{t})\ge A$. Since $c(U)$ is finite, we cannot make $A$ arbitrarily large; therefore, for $A$ large enough, $\beta(H_{t})$ must be at least $\hbar(L)$. Thus $\beta(U)\ge \hbar(L)$. Taking the supremum over $L$ and then taking the infimum over neighborhoods $U$ yields the desired result.
\end{proof}

\subsubsection{Spectral norm and non-interfering Lagrangians}
\label{sec:spectr-norm-hamilt}

Let us call disjoint open sets $U_{1}, U_{2}$ \emph{non-interfering} provided:
\begin{equation}\label{eq:max-formula-beta}
  \beta(H_{1,t}+H_{2,t})=\max\set{\beta(H_{1,t}),\beta(H_{2,t})}
\end{equation}
whenever $H_{i,t}$ is supported in $U_{i}$. The max-formula \eqref{eq:max-formula-beta} for boundary depth is proved for various pairs $U_{1},U_{2}$ in \cite{ganor-tanny-barricades}; for instance, if $U_{1},U_{2}$ are disjoint Darboux balls in $\C^{n}$, then $U_{1},U_{2}$ are non-interfering.

Let us say that a pair of closed Lagrangians $(L_{1},L_{2})$ is \emph{non-interfering} provided there is a non-interfering pair of neighborhoods $U_{1},U_{2}$ so that $L_{i}\subset U_{i}$. It is important to note that we do not require the $U_{i}$ to be arbitrarily small neighborhoods of $L_{i}$.

For example, if $L_{1},L_{2}$ are contained in disjoint Darboux balls in $\C^{n}$, then $L_{1},L_{2}$ are non-interfering. Also note that being non-interfering is invariant under Hamiltonian isotopy $(L_{1},L_{2})\mapsto (\varphi(L_{1}),\varphi(L_{2}))$.

We consider the following variant of the Lagrangian capacity.\footnote{See \cite[pp.\,3]{hind-arXiv-2024} for related discussion of packing two Lagrangians in a ball.} For an open subset $U\subset M$ we define:
\begin{equation*}
  \ell_{2}(U):=\sup \set{\min\set{\hbar(L_{1}),\hbar(L_{2})}:L_{1}\sqcup L_{2}\subset U\text{ is non-interfering}},
\end{equation*}
and extend this to all sets by outer regularity. Then:
\begin{theorem}\label{theorem:2L}
  For any subset $N\subset M$, we have: $$2\ell_{2}(N)\le \gamma(N);$$ in other words, a pair of non-interfering Lagrangians bounds the spectral diameter from below.
\end{theorem}
\begin{proof}
  Let $H_{1,t}\ge 0$ and $H_{2,t}\le 0$ be functions so that:
  \begin{equation*}
    \begin{aligned}
      &\max(H_{1,t})=\min(H_{1,t}|_{L_{1}})=A,\\
      &\min(H_{2,t})=\max(H_{2,t}|_{L_{2}})=-A,
    \end{aligned}
  \end{equation*}
  and suppose $H_{1,t},H_{2,t}$ are supported in non-interfering neighborhoods. Then, by the Hofer-norm upper bound to the boundary depth, applied to each function separately, we conclude:
  \begin{equation*}
    \beta(H_{1,t}+H_{2,t})\le A.
  \end{equation*}
  Therefore, if $A< \hbar(L_{1})$ and $A<\hbar(L_{2})$, we conclude:
  \begin{equation*}
    c([M],H_{1,t}+H_{2,t})\ge A\text{ and }c([M],\bar{H}_{1,t}+\bar{H}_{2,t})\ge A,
  \end{equation*}
  and hence $\gamma(H_{1,t}+H_{2,t})\ge 2A$. We can take $A$ any number smaller than $\min\set{\hbar(L_{1}),\hbar(L_{2})}$, and this yields the desired result.
\end{proof}

Theorem \ref{theorem:2L} and the result of \cite{alizadeh-atallah-cant-arXiv-2024} that $\gamma(B(1))=1$ implies:
\begin{theorem}\label{sec:cor-1-half}
  Let $L_{i}\subset \C^{n_{i}}$, $i=1,2$, be closed Lagrangians, and suppose that $L_{1}$ admits a nowhere-zero closed one-form and $n_{2}>0$. If $\varphi$ is a Hamiltonian diffeomorphism of $\C^{n_{1}+n_{2}}$ so that:
  \begin{equation*}
    \varphi(L_{1}\times L_{2})\subset B(1),
  \end{equation*}
  then $\hbar(L_{1}\times L_{2})<1/2$.

  In particular, if $n_{1}=n_{2}=1$, then we recover a special case of the result of \cite{viterbo-CRASP-1990,chekanov-math-z-1996,cieliebak-mohnke-inventiones-2018,hind-opshtein-CMH-2020} that $\varphi(\bd D(a_{1})\times \bd D(a_{2}))\subset B(1)$ holds for a Hamiltonian diffeomorphism $\varphi$ only if $\min\set{a_{1},a_{2}}<1/2$.
\end{theorem}

\emph{Remark}. As explained in \cite{viterbo-CRASP-1990,chekanov-math-z-1996}, the result that a torus: $$\bd D(a_{1})\times \dots\bd D(a_{n})$$ does not admit an exact isotopy into the interior of a ball $B(a)$ if $a_{1}\ge na$ is due to unpublished work of Floer-Hofer concerning a product formula for the Ekeland-Hofer capacities.

\begin{proof}
  Let $L_{1}^{\epsilon}$ be a small push-off of $L_{1}$ using the nowhere zero closed one-form. The Lagrangians $L=L_{1}\times L_{2}$ and $K_{s}=L_{1}^{\epsilon}\times (L_{2}+se_{1})$ are Hamiltonian isotopic, as pairs, for all $s$, and $L$ and $K_{s}$ are contained in disjoint balls for $s$ sufficiently large. Thus $L$ and $K_{0}$ are non-interfering. By taking $\epsilon$ small enough, we can ensure that $\hbar(K_{0})$ is as close to $\hbar(L_{1}\times L_{2})$ as desired. Thus, if $\varphi(L_{1}\times L_{2})\subset B(1)$, then Theorem \ref{theorem:2L} implies:
  \begin{equation*}
    2\hbar(L_{1}\times L_{2})< \gamma(B(1))=1,
  \end{equation*}
  where we use \cite{alizadeh-atallah-cant-arXiv-2024} in the last equality. Using split almost complex structures implies $\hbar(L_{1}\times L_{2})\ge \min\set{h(L_{1}),\hbar(L_{2})}$ which completes the proof.
\end{proof}

\emph{Remark.} Theorem \ref{sec:cor-1-half} fails if $n_{2}=0$, in which case $L_{2}=\C^{0}=\mathrm{pt}$. Indeed, we have $\bd D(a)\subset D(1)$ for all $r<1$ and $\hbar(\bd D(a))=a$ can be made larger than $1/2$. This is because $\bd D(a)$ and $\bd D(a+\epsilon)$ are interfering.

\subsection{Contrast with the hypersurface case}
\label{sec:contr-with-hypersurface}

We also prove the following bound on the spectral capacity of certain contact-type hypersurfaces:

\begin{theorem}\label{theorem:5}
  Let $(M,\alpha)$ be a Liouville manifold, and let $N\subset M$ be a compact hypersurface so that $\ell$ restricts to $N$ as a contact form (we say that $N$ is of restricted contact type). Then $c(N)$ is bounded from below by the minimal period of the Reeb flow of $\alpha|_{N}$.
\end{theorem}

This result is proved by a direct analysis of the spectrum of actions appearing in a certain specific systems supported near $N$. The idea of directly analyzing the spectrum of orbits is not new, and a more refined analysis in \cite{ginzburg-duke-2007} proves: {\itshape if $(M,\omega)$ is an aspherical symplectic manifold, and $N\subset M$ is a coisotropic submanifold which is stable and displaceable then the spectral capacity of $N$ is positive.} We include the proof of Theorem \ref{theorem:5} for the reader's convenience.

It is interesting to recall the following question posed in \cite[\S3.3.4]{ginzburg-duke-2007}:
\begin{question}[Ginzburg]\label{question:ginzburg}
  Does it hold that $c(\Sigma)>0$ for every closed hypersurface $\Sigma\subset \R^{2n}$?
\end{question}

\subsection{Beyond stable coisotropic submanifolds}
\label{sec:gener-gener-cois}

Most of the existing literature focus on the case when $N$ is \emph{stable} in the sense of \cite{bolle-CRAS-1996,bolle-math-Z-1998}; see, e.g., \cite{ginzburg-duke-2007} and \cite{dragnev-CPAM-2008,kerman-j-mod-dyn-2008,gurel-IMRN-2010,usher-israel-j-math-2011,kang-IMRN-2013,ginzburg-gurel-math-Z-2015}. For other works on coisotropic submanifolds which do not assume stability, see \cite{moser-acta-1978,hofer-PRS-edinburgh-1990,ziltener-IMRN-2017,lisi-rieser-jsg-2020}.

It is noteworthy that \cite{ginzburg-gurel-math-Z-2015} construct non-stable hypersurfaces $N$ which can \emph{leaf-wise displaced}\footnote{\emph{leaf-wise displaced} means no pair $(x,\varphi_{1}(x))$, $x\in N$, lies on the same leaf of the characteristic foliation of $N$} by Hamilotonian isotopies $\varphi_{t}$ with arbitrarily small Hofer length. This suggests that any positive solution of Question \ref{question:ginzburg} in the general case will not be based on leaf-wise intersection points.

However, the spectral capacity of the hypersurfaces constructed in \cite[pp.\,993]{ginzburg-gurel-math-Z-2015} is uniformly positive; indeed, the hypersurfaces contain a common Lagrangian, and one can then appeal to the lower bound $\ell(N)\le c(N)$.

Before ending this section, we mention a related question:
\begin{question}\label{question:8}
  Does it hold that $\ell(N)>0$ for every compact coisotropic submanifold $N$? In particular, does it hold that $\ell(\Sigma)>0$ for every closed hypersurface in $\Sigma\subset \R^{2n}$?
\end{question}
This is relevant because a positive answer to Question \ref{question:8} implies a positive answer to Question \ref{question:ginzburg}.

\subsubsection{Sets which contain Lagrangian submanifolds}
\label{sec:cont-lagr-subm}

One way to give a positive answer to Question \ref{question:8} for $N$ is to find a closed Lagrangian submanifold $L\subset N$. However, the following shows this is not always possible:
\begin{proposition}
  There exist closed starshaped hypersurfaces $N\subset \R^{4}$ which do not contain closed Lagrangians. Indeed, any closed starshaped hypersurface whose Reeb flow is ergodic does not contain any closed Lagrangian. In fact, we only require that the Reeb flow has a dense trajectory.\footnote{An ergodic flow has many dense trajectories; see, e.g., \cite[Lemma 11]{casals-spacil-j-topol-anal-2016}.}
\end{proposition}
Here we recall that a measure preserving dynamical system is \emph{ergodic} provided all invariant sets have either full measure or zero measure. The measure used for Reeb flows is the volume measure of the contact form.

\begin{proof}
  The argument is quite simple: any Lagrangian $L^{2}\subset N^{3}$ is a \emph{dividing hypersurface} because $N$ is diffeomorphic to a sphere. Moreover, by definition of the characteristic foliation, the Reeb vector field must be tangent to $L$. In particular $L$ divides $N$ into two disjoint invariant subsets, both of which have nonzero measure. This contradicts ergodicity. Since the sets are in fact open, this also contradicts the existence of a dense trajectory.
\end{proof}

As to the construction of such ergodic Reeb flows, we refer the reader to \cite[\S4.2]{casals-spacil-j-topol-anal-2016} which constructs an ergodic Reeb flow on the standard contact sphere $S^{2n+1}$ for every $n$ using \cite[Theorem A]{katok-izv-akad-nauk-1973}. The Reeb flow constructed using \cite[Theorem A]{katok-izv-akad-nauk-1973} is generated by a homogeneous Hamiltonian $H:\R^{2n}\setminus \set{0}\to \R$ which can be taken arbitrarily close to $H_{0}=\pi\abs{z}^{2}$ in the $C^{\infty}_{\mathrm{loc}}$ topology. In particular, the hypersurface $N=H^{-1}(1)$ is arbitrarily close to the standard 3-sphere, and $N$ does not contain a closed Lagrangian.

\emph{Remark.} There has recently been much work concerning closed invariant subsets in hypersurfaces in $\R^{4}$, e.g., \cite{fish-hofer-annals-2023,c-gardiner-prasad-arXiv-2024}. In general, if $X\subset N$ is a subset of a coisotropic submanifold, then we say $X$ is an \emph{invariant subset} provided $x\in X$ implies the entire characteristic leaf through $x$ is contained in $X$. We say $X$ is \emph{non-trivial} provided $X\ne \emptyset$ and $X\ne N$. It seems to be an interesting question whether a closed coisotropic submanifold $N^{d}$, with $d>n$, always contains a non-trivial closed invariant subset. This has been answered in the affirmative for $N^{3}\subset \R^{4}$ by \cite{fish-hofer-annals-2023}. This question generalizes the existence of a closed Lagrangian submanifold $L\subset N$, since $L$ is always a non-trivial closed invariant subset.\footnote{Note that if $d=n$ and $N^{d}$ is a connected coisotropic submanifold then there is a single leaf of its characteristic foliation.}

\subsubsection{Toric sets}
\label{sec:toric-sets}

An obvious class of coisotropic submanifolds for which Question \ref{question:8} has a positive answer are the \emph{toric sets}, i.e., those obtained as preimages $N=\mu^{-1}(\Gamma)$ of subsets $\Gamma$ under a moment map (assuming at least one value in $\Gamma$ is a regular value with non-empty fiber).

As a special case, if $\mu:\C^{n}\to \R^{n}$ is the standard moment map $\mu(z)=\pi\abs{z}^{2}$, then one easily deduces:
\begin{equation*}
  \hbar(\mu^{-1}(a_{1},\dots,a_{n}))\ge \min\set{a_{1},\dots,a_{n}},
\end{equation*}
and, in this setting,
\begin{equation}\label{eq:inequality}
  c(\mu^{-1}(\Gamma))\ge \ell(\mu^{-1}(\Gamma))\ge \max\set{\min\set{a_{1},\dots,a_{n}}:a\in \Gamma}.
\end{equation}
When $n=2$ and $\Gamma$ is the standard simplex (so $\mu^{-1}(\Gamma)=B(1)$ is the ball of capacity $1$), the first and last terms in \eqref{eq:inequality} differ by a factor of $n$. In \cite{swoboda-ziltener-JSG-2013} it is shown that $\ell(B(1))\ge 1/2$. One realizes this lower bound using the Lagrangian:
\begin{equation}\label{eq:special-lagrangian}
  L=\set{zq:z\in S^{1}\text{ and }q\in \bd B(1)\cap \R^{n}}
\end{equation}
considered in \cite{weinstein-book-1977,audin-CMH-1988,polterovich-math-Z-1991}.

On the other hand, our Theorem \ref{sec:cor-1-half} proves the upper bound $\hbar(L)<1/2$ whenever $L=\varphi(L_{1}\times L_{2})$ if $\varphi$ is a Hamiltonian diffeomorphism and $L_{1},L_{2}$ satisfy certain hypotheses. This suggests the question:
\begin{question}
  What is the exact value of $\ell(B(1))$?
\end{question}

If one considers the variant of the Lagrangian capacity which considers the rationality constants $\rho(T)$ of Lagrangian tori $T$ (see \cite{cieliebak-mohnke-inventiones-2018,pereira-arXiv-2022,gutt-pereira-ramos-arXiv-2022}), then, for a large class of domains $\Gamma$, it holds that:
\begin{equation}\label{eq:cm-pereira}
  \sup\set{\rho(T):T\subset \mu^{-1}(\Gamma)}=\max\set{\min\set{a_{1},\dots,a_{n}}:a\in \Gamma}.
\end{equation}
The equality \eqref{eq:cm-pereira} was proved when $\Gamma$ is the standard simplex in \cite{cieliebak-mohnke-inventiones-2018}; see also \cite{viterbo-CRASP-1990,chekanov-math-z-1996} and unpublished work of Floer-Hofer which proves \eqref{eq:cm-pereira} holds when $\Gamma$ is a simplex in terms of the capacities from \cite{ekeland-hofer-math-z-1990} provided one restricts the supremum to only those $T$ symplectomorphic to a Clifford torus (i.e., there is an ambient symplectomorphism of $\R^{2n}$ taking a Clifford torus to $T$).

In dimension $n=2$, the \cite{cieliebak-mohnke-inventiones-2018} show the rationality constant of any Lagrangian contained in the ball $B(1)$ is at most $1/2$ (such Lagrangians include tori and certain negatively curved non-orientable surfaces) .

In related work, \cite[\S7]{cote-JSG-2020} studies a refinement $\rho_{2}(T)$ of the rationality constants of tori which considers only the symplectic areas of disks with Maslov number $2$; \cite{cote-JSG-2020} proves the equality \eqref{eq:cm-pereira} in dimension $n=2$ when $\Gamma$ is a rectangle and with $\rho$ replaced by $\rho_{2}$; see also \cite{charette-AGT-2015}.

\subsection{The Hofer diameter of small open sets}
\label{sec:hofer-diameter-small}

It is well-known that the spectral norm of $\varphi_{t}$ is bounded from above by the Hofer norm of $\varphi_{t}$ computed in the universal cover. It is therefore natural to consider the following quantity:
\begin{equation}\label{eq:hofer-diameter-version}
  \sup\set{\norm{\varphi_{t}}_{\mathrm{Hofer}}:\varphi_{t}\text{ is supported in }U}
\end{equation}
as a Hofer analogue of $\gamma(U)$; here $U$ is open and:
\begin{equation*}
  \norm{\varphi_{t}}_{\mathrm{Hofer}}:=\inf\set{\int_{0}^{1}\max_{M} H_{t}-\min_{M} H_{t}d t:H_{t}\sim \varphi_{t}},
\end{equation*}
where $H_{t}\sim \varphi_{t}$ means the time-1 isotopy generated by $H_{t}$ is homotopic to $\varphi_{t}$ with fixed endpoints. However:
\begin{theorem}\label{theorem:7}
  For any non-empty open set $U$ in a compact semipositive symplectic manifold $M$, the quantity in \eqref{eq:hofer-diameter-version} is infinite.
\end{theorem}
The argument is based on the existence of a measurement $m(\varphi_{t})$ satisfying:
\begin{enumerate}[label=(M\arabic*)]
\item\label{M1} $m(\varphi_t)\le \norm{\varphi_{t}}_{\mathrm{Hofer}}$,
\item\label{M2} $m(\varphi_{t})=\mathrm{Vol}(M)^{-1}\mathrm{Cal}(\varphi_{1})$ if $\varphi_{t}$ is supported in a displaceable Darboux ball,
\end{enumerate}
where $\mathrm{Cal}(\varphi)$ is the \emph{Calabi invariant} for Hamiltonian diffeomorphisms of an exact symplectic manifold; see, e.g., \cite{mcduff-salamon-book-2012}.

Such a measurement can be constructed as a homogenization of the spectral invariants of the mean-zero Hamiltonian $H_{t}$ generating $\varphi_{t}$. To be precise: let $H_{t}^{k}$ be the mean-zero Hamiltonian generating $\varphi_{t}^{k}$, and define:
\begin{equation}\label{eq:measurement}
  m(\varphi_{t})=-\lim_{k\to 0}\frac{c([M],H_{t}^{k})}{k}.
\end{equation}
We refer the reader to \cite{entov-poltero-IMRN-2003} for a similar quantity satisfying a similar Calabi property. The straightforward verification that $m$ satisfies \ref{M1} and \ref{M2} is recalled in \S\ref{sec:proof-theor-7}.

Since one can find systems supported in any ball with arbitrarily large Calabi invariant, Theorem \ref{theorem:7} easily follows easily from \ref{M1} and \ref{M2}.

It is interesting to ask whether Theorem \ref{theorem:7} holds if one does not work in the universal cover of the group of Hamiltonian diffeomorphisms. This is of course related to the long-standing open question whether the Hofer diameter of every compact symplectic manifold is infinite; see \cite{ostrover-CCM-2003,mcduff-CMH-2010} for related discussion.

It is also interesting to consider to what extent Theorem \ref{theorem:7} holds if one replaces $M$ by an open symplectic manifold. It cannot hold in general since the results of \cite{sikorav-pisa-1990} and \cite{hofer-zehnder-book-1994,burago-ivanov-polterovich-2008,brandenbursky-kedra-ann-math-que-2017,khanevsky-ziltener-dga-2022} show that:
\begin{proposition}
  Any compact set in $\R^{2}\times M$ with the product symplectic structure has a bounded Hofer diameter in the universal cover.\hfill$\square$
\end{proposition}
The argument uses the fact that any precompact open set $U$ is displaceable with infinite packing number, i.e., one can find Hamiltonian diffeomorphisms $\psi_{1},\psi_{2},\dots,$ of $\R^{2}\times M$ so that the images $\psi_{i}(U)$, $i=1,2,\dots$, are pairwise disjoint. See \cite{polterovich-shelukhin-compositio-2023} for further discussion of such packings.

\subsection{Contact geometry speculations}
\label{sec:contact-geometry-speculations}

It is natural to wonder which of these results has analogues in contact geometry. Let us briefly speculate on this.

Using spectral invariants from contact Floer cohomology, as in \cite{cant-arXiv-2023,djordjevic-uljarevic-zhang-arXiv-2023}, one can associate a spectral invariant $c_{\alpha}(\phi_{t})$ to any contact isotopy $\phi_{t}$ of the ideal boundary of a Liouville manifold $W$ whose symplectic cohomology is non-vanishing. The definition depends on choice of contact form $\alpha$ on the ideal boundary. One can then define capacities as above. However, these spectral invariants are unfortunately not invariant under the conjugation action of the contactomorphism group. The resulting capacities will generally fail to be invariant under the contactomorphism group.

One way to extract invariant measurements via contact Floer cohomology is to appeal to more discrete invariants. For example, \cite{cant-uljarevic-arxiv-2024} defines a functor from the category of domains in $Y$ to the category of vector spaces:
\begin{equation*}
  \Omega\subset Y\mapsto Q(\Omega)\in \mathrm{Vect}(\Z/2).
\end{equation*}
Moreover, every contactomorphism $\varphi$ of $Y$ which extends to a symplectomorphism of $W$ induces a natural isomorphism of this functor.

With this structure one can, e.g., define invariant measurements such as:
\begin{equation*}
  q(\Omega)=\text{rank of $Q(\Omega)\to Q(Y)$}.
\end{equation*}
One can extend such an invariant to all subsets of $Y$ by outer regularity, and it is interesting to wonder which submanifolds the resulting capacity is sensitive to.

For other approaches to defining invariant measurements in contact geometry see, e.g., \cite{givental-advsov-1990,eliashberg-JAMS-1991,eliashberg-kim-polterovich-2006,sandon-ann-inst-four-2011,zapolsky-IMRN-2013,albers-merry-JSG-2018,allais-arlove-arXiv-2023}

\subsection{Acknowledgements}
\label{sec:acknowledgements}

The authors wish to thank Yong Geun Oh, Georgios Dimitroglou Rizell, and Jungsoo Kang for illuminating discussions during the 2024 IBS-CGP MATRIX workshop in Pohang, South Korea, when the idea for this paper began. The authors also wish to thank Marcelo Atallah, Habib Alizadeh, Octav Cornea, Yaniv Ganor, Eric Kilgore, Rémi Leclerq, Egor Shelukhin, Igor Uljarević, and Michael Usher for useful conversations surrounding the topics in this paper. The first named author is supported in his research by funding from the ANR project CoSy. The second named author is partially supported by National Key R\&D Program of China No.~2023YFA1010500, NSFC No.~12301081, NSFC No.~12361141812, and USTC Research Funds of the Double First-Class Initiative.

\section{Proofs}
\label{sec:proofs}

\subsection{Proof of Theorem \ref{theorem:1}}
\label{sec:proof-theorem-1}

The proof has two main steps. The first step is to relate the spectral capacity to the stable displacement energy. The second step is to prove compact nowhere coisotropic submanifolds have neighborhoods with arbitrarily small stable displacement energy.

\subsubsection{Stable displacement energy}
\label{sec:stable-displ-energy}

We say that a compact set $K\subset M$ is \emph{stably displaceable} provided that: $$K\times S^{1}\subset M\times T^{*}S^{1}$$ is displaceable via a Hamiltonian isotopy. In this case, we say that the infimal Hofer length of an isotopy in $M\times T^{*}S^{1}$ displacing $K\times S^{1}$ is its \emph{stable displacement energy}, denoted $sde(K)$. For an open set $U$, we define:
\begin{equation*}
  sde(U)=\sup\set{sde(K):K\subset U\text{ compact}}.
\end{equation*}

The first goal in this section is to prove:
\begin{theorem}\label{theorem:stable-displacement}
  The spectral capacity $c(U)$ of an open set $U\subset M$ is bounded from above by the stable displacement energy $sde(U)$.
\end{theorem}
A related result appears in \cite[\S6]{entov-polterovich-compositio-2009} and \cite[\S3.6]{borman-JSG-2012}, although our result is proved by relating spectral invariants in $M$ with $M\times T^{2}$ while \cite{borman-JSG-2012} uses $M\times S^{2}$.

The key will be the following result involving the torus: $$T_{a}=\R/\Z\times \R/2a\Z,$$ whose coordinates are labelled $x,y$.
\begin{lemma}\label{lemma:technical-estimate}
  Let $H_{t}$ be a compactly supported time-dependent Hamiltonian function on $M$, and let $\rho:\R/2a\Z\to \R$ be a smooth function so that $\rho(y)=1$ for $y$ in a neighborhood of the pair of antipodal points $\set{0,a}$. Then:
  \begin{equation*}
    c([M\times T_{a}],\rho(y)H_{t}+k\cos(\pi y/a))=c([M\times T_{a}],H_{t}+k\cos(\pi y/a))
  \end{equation*}
  holds for $k$ sufficiently large.
\end{lemma}
\begin{proof}
  The idea is to consider the 1-parameter family of Hamiltonian functions on $M\times T_{a}$, parametrized by $\tau\in [0,1]$, given by:
  \begin{equation*}
    G_{\tau}=(1-\tau+\tau\rho(y))H_{t}+k\cos(\pi y/a),
  \end{equation*}
  and to show that, for $k$ sufficiently large, the closed contractible orbits of the system $\Phi_{\tau,t}$ generated $G_{\tau}$ and their actions are independent of $\tau$. It will then follow from the continuity and spectrality\footnote{We should assume that $M$ is rational in order to appeal to spectrality. In \S\ref{sec:rati-assumpt-cont} we explain how to drop the rationality assumption.} properties for spectral invariants that $c([M\times T_{a}];G_{\tau})$ is independent of $\tau$.

  We compute:
  \begin{equation*}
    \d G_{\tau}=(1-\tau+\tau\rho(y))\d H_{t}+(\tau\rho'(y)H_{t}-\pi a^{-1} k\sin(\pi y/a))\d y.
  \end{equation*}
  Now we use the fact that $\rho'(y)=0$ for $y$ in a neighborhood of $\set{0,a}$. In particular, for $k$ sufficiently large,
  \begin{equation}\label{eq:requirement-on-k}
    \tau\rho'(y)H_{t}-\pi a^{-1} k\sin(\pi y/a)=0\iff y\in \set{0,a}.
  \end{equation}
  Since the symplectic vector field associated to $\d y$ is $\bd_{x}$, and $X_{H_{t}}$ is tangent to the level sets $M\times\set{(x,y)}$, it follows that:
  \begin{equation*}
    \text{contractible orbits of $\Phi_{\tau,t}$}=\set{(\gamma,x,y):
      \begin{aligned}
        &\gamma\text{ is an orbit of }X_{H_{t}},\\
        &x\in \R/\Z\text{ and }y\in \set{0,a}
      \end{aligned}
    }.
  \end{equation*}
  It remains to prove the actions of these contractible orbits are independent of $\tau$. In other words, it remains to prove that the integral of $G_{\tau,t}$ over the above orbits is independent of $\tau$. This follows immediately since we assume $\rho(y)=1$ in a neighborhood of $\set{0,a}$. This completes the proof.
\end{proof}

Using this result, and the K\"unneth formula for spectral invariants proved in \cite[Theorem 5.1]{entov-polterovich-compositio-2009} (see also \S\ref{sec:prod-form-spectr}), we obtain:
\begin{lemma}\label{lemma:spectral-capacity-stabilization}
  Let $H_{t},\rho$ be as in Lemma \ref{lemma:technical-estimate}. Then:
  \begin{equation*}
    c([M],H_{t})\le c([M\times T_{a}],\rho(y)H_{t}).
  \end{equation*}
  Consequently, the spectral capacity of $U\subset M$ is bounded from above by the spectral capacity of $U\times \Gamma\subset M\times T_{a}$ where $\Gamma=\set{y=0}\cup \set{y=a}$.
\end{lemma}
\begin{proof}
  We begin by using Lemma \ref{lemma:technical-estimate} (with $k$ taken large enough) and the K\"unneth formula from \cite[Theorem 5.1]{entov-polterovich-compositio-2009} to obtain:
  \begin{equation*}
    \begin{aligned}
      c([M],H_{t})+c([T_{a}],k\cos(\pi y/a))&=c([M\times T_{a}],H_{t}+k\cos(\pi y/a))\\
      &=c([M\times T_{a}],\rho(y)H_{t}+k\cos(\pi y/a)).
    \end{aligned}
  \end{equation*}
  Since $\rho(y)H_{t}$ and $k\cos(\pi y/a)$ are Poisson-commuting, we can apply the triangle inequality\footnote{We do not prove the well-known triangle inequality for spectral invariants in this paper; see, e.g., \cite{alizadeh-atallah-cant-arXiv-2023} for some discussion in the convex-at-infinity case.} for spectral invariants to obtain:
  \begin{equation*}
    \begin{aligned}
      &c([M\times T_{a}],\rho(y)H_{t}+k\cos(\pi y/a))\\
      &\hspace{2cm}\le c([M\times T_{a}],\rho(y)H_{t})+c([M\times T_{a}],k(\cos(\pi y/a))).\\
      &\hspace{2cm}= c([M\times T_{a}],\rho(y)H_{t})+c([T_{a}],k(\cos(\pi y/a))),
    \end{aligned}
  \end{equation*}
  where we have used \cite[Theorem 5.1]{entov-polterovich-compositio-2009} again in the last equality. Combining everything and cancelling the common term $c([T_{a}],k(\cos(\pi y/a)))$ yields the desired result.
\end{proof}

The final lemma used to prove Theorem \ref{theorem:stable-displacement} relates displacement energy in $M\times T^{*}S^{1}$ and $M\times T_{a}$, for $a$ sufficiently large.
\begin{lemma}\label{lemma:displacement-energy-Ta}
  Let $K\subset M$ be a compact set. The displacement energy of $K\times S^{1}$ in $M\times T^{*}S^{1}$ equals:
  \begin{equation*}
    \mathrm{inf}\set{\text{displacement energy of $K\times \Gamma_{a}$ in $M\times T_{a}$}:a>0},
  \end{equation*}
  where $\Gamma=\set{y=0}\cup \set{y=a}$. In particular, if $sde(K)<\epsilon$, then $K\times \Gamma_{a}$ has displacement energy in $M\times T_{a}$ less than $\epsilon$ for $a$ large enough.
\end{lemma}
\begin{proof}
  The proof is a straightforward construction. See also \cite[\S6]{entov-polterovich-compositio-2009}.
\end{proof}
Theorem \ref{theorem:stable-displacement} follows from Lemmas \ref{lemma:spectral-capacity-stabilization} and \ref{lemma:displacement-energy-Ta} and the well-known energy-capacity inequality applied to $M\times T_{a}$; see \cite{schwarz-pacific-j-math-2000,oh-2005-duke,frauenfelder-ginzburg-schlenk,ginzburg-2005-weinstein,usher-CCM-2010} for the proof of the energy-capacity inequality.

\begin{proof}[Proof of Theorem \ref{theorem:stable-displacement}]
  For any function $H_{t}$ with compact support $K\subset U$, Lemma \ref{lemma:displacement-energy-Ta} implies we can displace a neighborhood of $K\times \Gamma$ in $M\times T_{a}$ (for large enough $a$) by a Hamiltonian isotopy whose Hofer length is at most $sde(U)+\epsilon$, where $\epsilon$ is an arbitrarily small number. Then Lemma \ref{lemma:spectral-capacity-stabilization} and the energy capacity inequality yield:
\begin{equation*}
  c([M],H_{t})\le sde(U)+\epsilon.
\end{equation*}
Infimizing over $\epsilon$ and then supremizing over systems $H_{t}$ yields the desired conclusion.
\end{proof}

\subsubsection{Nowhere coisotropic submanifolds}
\label{sec:nowh-cois-subm}

The next result is:
\begin{theorem}\label{lemma:coisotropic-sde}
  Let $S\subset M$ be a compact nowhere coisotropic submanifold. For any $\epsilon>0$, $S$ has a neighborhood $U$ with $sde(U)< \epsilon$.\hfill$\square$
\end{theorem}
This is proved in \cite{gurel-CCM-2008}. We briefly recall the argument. The key idea is the following lemma:
\begin{lemma}\label{lemma:key-idea}
  Let $S$ be a submanifold (not necessarily nowhere coisotropic) and suppose that $X$ is a nowhere vanishing section of $TS^{\perp\omega}$ which admits a Lyapunov function $F$ defined on a neighborhood of $S$, i.e., $\d F(X)>0$ holds along $S$. Then the Hamiltonian vector field $X_{F}$ is nowhere tangent to $S$.

  If $X_{F}$ is a complete vector field, then for each $\epsilon>0$, $S$ has a neighborhood $U$ so that each compact set $K\subset U$ can be displaced by Hamiltonian isotopy with Hofer length at most $\epsilon$.
\end{lemma}

\begin{proof}
  We compute $\omega(X,X_{F})>0$, and thus $X_{F}$ cannot be tangent to $S$.

  Suppose that $X_{F}$ is a complete vector field. Replace $F$ by $\epsilon\pi^{-1}\arctan(F)$, so that $\max F-\min F<\epsilon$. Clearly $F$ is still a Lyapunov function for $X$, and hence $X_{F}$ is still transverse to $S$. Moreover $X_{F}$ is still a complete vector field.

  By the parametric transversality theorem, there exists some time $t_{0}\in (0,1)$ so that $\phi_{t_{0}}(S)$ and $S$ are disjoint. In particular, there are disjoint open sets $U_{1},U_{2}$ around $S$ and $\phi_{t_{0}}(S)$ so that $\phi_{t_{0}}(U_{1})\subset U_{2}$. Then $U=U_{1}$ is the desired open set, since any compact set in $U$ will be displaced by $\phi_{t_{0}}$, which has Hofer length at most $\epsilon$.
\end{proof}

Let $S$ be nowhere coisotropic. Then $S\times S^{1}\subset M\times TS^{1}$ is still nowhere coisotropic, and $T(S\times S^{1})^{\perp \omega}$ contains $TS^{1}\subset T(M\times S^{1})$. In particular, $T(S\times S^{1})^{\perp}$ admits a non-vanishing section $Z$. The strategy is to replace $Z$ by another non-vanishing section $X$ which admits a Lyapunov function $F$. The previous lemma then implies that $S\times S^{1}$ has neighborhoods with arbitrarily small stable displacement energy, yielding Theorem \ref{lemma:coisotropic-sde}.

To achieve this one uses the criterion for existence of Lyapunov functions given in \cite[Theorem II.26]{sullivan-inventiones-1976}; see also \cite{laudenbach-sikorav-IMRN-1994}, and \cite[\S1.4]{gromov-springer-1986}. Applying this criterion, \cite{gurel-CCM-2008} concludes:

\begin{lemma}\label{lemma:coisotropic-lyapunov}
  Let $S\subset M$ be a compact nowhere coisotropic submanifold and suppose that $TS^{\perp\omega}$ has a non-vanishing section; then there is a non-vanishing section $X$ of $TS^{\perp\omega}$ and a smooth function $F$ so that $\d F(X)>0$.\hfill$\square$
\end{lemma}
Together with Lemma \ref{lemma:key-idea}, we conclude Theorem \ref{lemma:coisotropic-sde}, completing the proof of Theorem \ref{theorem:1}.

Interestingly enough, if $L$ is an open Lagrangian, then $L$ admits a nowhere vanishing gradient vector field $X$, so $X$ admits a Lyapunov function. Since $TL^{\perp\omega}=TL$, there are Hamiltonian vector fields which are nowhere tangent to $L$. This shows that being nowhere coisotropic is not a necessary condition to be transverse to a Hamiltonian vector field.

\subsection{Proof of Lemma \ref{lemma:3}}
\label{sec:proof-lemma-3}

Let $H_{t},L,$ be as in the statement, and introduce $J\in \mathscr{J}$ so that:
\begin{equation*}
  \beta(H_{t})+\int_{0}^{1}\max(H_{t})-\min(H_{t}|_{L})dt < \hbar(L,J).
\end{equation*}

Consider the moduli space $\mathscr{M}$ of half-infinite Floer cylinders:
\begin{equation}\label{eq:defn-mathfrak-e}
  \left\{
    \begin{aligned}
      &u:(-\infty,0]\times \R/\Z\to M,\\
      &\bd_{s}u+J(u)(\bd_{t}u-X_{\delta,t}(u))=0,\\
      &u(0,\R/\Z)\in L,\\
      &u(0,0)=\mathrm{pt},
    \end{aligned}
  \right.
\end{equation}
where $\mathrm{pt}$ is a fixed basepoint in $L$, and $H_{\delta,t}$ is a $C^{2}$-small perturbation of $H_{t}$ used to achieve non-degeneracy of the time-1 orbits and transversality of the relevant moduli spaces; here $X_{\delta,t}$ is the Hamiltonian vector field of $H_{\delta,t}$. We continue to suppose $H_{\delta,t}$ satisfies the hypotheses of Lemma \ref{lemma:3}.

For convenience in the proof, introduce the abbreviations:
\begin{equation*}
  \beta:=\beta(H_{\delta,t}),\hspace{.25cm}E_{+}=\int_{0}^{1}\max(H_{\delta,t})dt,\hspace{.25cm}E_{-}=\int_{0}^{1}\min(H_{\delta,t}|_{L})dt,
\end{equation*}
and $\hbar=\hbar(L,J)$. We suppose $\beta+E_{+}-E_{-}+2\epsilon<\hbar$ for a small $\epsilon>0$.

Let us say that $u\in \mathscr{M}$ is \emph{admissible} provided:
\begin{enumerate}
\item there is a capped orbit $(x,v)$ with action less than $\beta+E_{+}+2\epsilon$,
\item the left asymptotic of $u$ is the orbit $x$,
\item the concatentation $v\# u$ forms a disk with boundary on $L$ with zero symplectic area.
\end{enumerate}

Let $\mathrm{CF}_{<\beta+E_{+}+2\epsilon}(H_{\delta,t},J)$ be the subcomplex generated by the capped orbits of $X_{\delta,t}$ whose actions are less than $\beta+E_{+}+2\epsilon$. Define a map:
\begin{equation*}
  \mathfrak{e}:\mathrm{CF}_{<\beta+E_{+}+2\epsilon}(H_{\delta,t},J)\to \Z/2,
\end{equation*}
by counting the rigid elements in $\mathscr{M}$ as follows: $\mathfrak{e}(x,v)$ is the number of rigid elements in $\mathscr{M}$ whose left asymptotic of $u$ is the orbit $x$ and $v\# u$ forms a disk with boundary on $L$ with zero symplectic area. It is clear that only admissible elements contribute to this count.

The key estimate is:
\begin{proposition}
  The admissible curves $u\in \mathscr{M}$ have Floer energies bounded from above by $\beta+E_{+}-E_{-}+2\epsilon<\hbar$.
\end{proposition}
\begin{proof}
  If $u\in \mathscr{M}$ is admissible, with left asymptotic $x$, then:
  \begin{equation*}
    E(u)\le \omega(u)+\int_{x} H_{\delta,t}-\int_{u(0,t)}H_{\delta,t}=\text{(action of $(x,v)$)}-\int_{u(0,t)}H_{\delta,t},
  \end{equation*}
  where the capping $v$ is such that $v\# u$ has zero symplectic area. Since we assume that $(x,v)$ has action less that $\beta+E_{+}+2\epsilon$, and $\int_{0}^{1}H_{\delta,t}|_{L}dt$ is at least $E_{-}$, we conclude the desired upper bound.
\end{proof}
This a priori energy bound implies the piece of $\mathscr{M}$ used to define $\mathfrak{e}$, and to prove it is a chain map, is compact up to breaking of Floer cylinders; disk bubbling on $L$ (or sphere bubbling in $M$) cannot occur since the energies are below the bubbling threshold (bearing in mind that $\hbar$ is the minimal area a $J$-holomorphic disk or sphere).

Standard Floer theoretic arguments then prove $\mathfrak{e}:\mathrm{CF}_{<\beta+E_{+}+2\epsilon}\to \Z/2$ is a chain map, for a generic perturbed system $H_{\delta,t}$.

The next step in the argument is to show that some representative of the unit is sent by $\mathfrak{e}$ to $1$:
\begin{proposition}\label{prop:PSS-e}
  Let $\zeta\in \mathrm{CF}_{<E_{+}+\epsilon}$ be the cycle representing the unit obtained by the standard PSS map (as recalled below). Then $\mathfrak{e}(\zeta)=1$.
\end{proposition}
\begin{proof}
  We begin by briefly recalling the PSS cycle of \cite{piunikhin-salamon-schwarz-1996} which represents the unit element; further details are given in \S\ref{sec:pss-spectr-invar}. It is defined by counting the rigid finite-energy solutions to:
  \begin{equation}\label{eq:PSS}
    \left\{
      \begin{aligned}
        &v:\C\to M\text{ smooth},\\
        &u=v(e^{2\pi (s+it)}),\\
        &\bd_{s}u+J(u)(\bd_{t}u-\beta(s)X_{\delta,t}(u))=0,
      \end{aligned}
    \right.
  \end{equation}
  where $\beta(s)$ is a standard cut-off function satisfying $\beta(s)=0$ for $s\le 0$ and $\beta(s)=1$ for $s\ge 1$. Each such rigid solution $v$ is considered as a capping of the asymptotic orbit, and therefore the count of rigid solutions, denoted $\zeta$, is valued in the complex $\mathrm{CF}(H_{\delta,t},J)$.

  A standard energy estimate,\footnote{Some care is needed in the open case; one needs to appeal to the maximum principle to conclude this energy estimate. See \cite[\S5.1]{frauenfelder-schlenk-IJM-2007} for further discussion.} as in \cite{schwarz-pacific-j-math-2000}, implies that the action of the resulting capped orbit is bounded by $E_{+}$, and hence the PSS element $\zeta$ is valued in $\mathrm{CF}_{<E_{+}+\epsilon}(H_{\delta,t},J)$.

  It remains to prove that $\mathfrak{e}(\zeta)=1$. A similar argument appears in \cite{alizadeh-atallah-cant-arXiv-2023}. We will consider the parametric moduli space of pairs $(R,w)$ where $w:D(1)\to M$ is smooth and has zero symplectic area and solves:
  \begin{equation*}
    \left\{
      \begin{aligned}
        &u=w(e^{2\pi (s+it)})\text{ (so }u:(-\infty,0]\times \R/\Z\to M),\\
        &\bd_{s}u+J(u)(\bd_{t}u-\beta(s+R)X_{\delta,t}(u))=0,\\
        &u(0,t)\in L\text{ and }u(0,0)=\mathrm{pt},
      \end{aligned}
    \right.
  \end{equation*}
  and $R\in \R$. This parametric moduli space has two non-compact ends. One end is when $R<0$, in which case $w$ is simply a $J$-holomorphic disk, with zero symplectic area, and hence must equal the point $\mathrm{pt}$. The other is when $R\to \infty$, in which case $w$ ``splits'' in the Floer theoretic sense into a configurations $(v,u)$ where $v$ solves \eqref{eq:PSS} and $u$ solves the equation \eqref{eq:defn-mathfrak-e} defining $\mathfrak{e}$. The count of such rigid configurations is the composition $\mathfrak{e}(\zeta)$.

  The usual energy estimate for Floer cylinders implies that any solution $w$ has energy:
  \begin{equation*}
    \int_{0}^{1}\int_{-\infty}^{0}\omega(\bd_{s}u,\bd_{t}u-\beta(s+R)X_{\delta,t}(u))ds dt\le E_{+}-E_{-}<\hbar,
  \end{equation*}
  and hence the 1-dimensional parametric moduli space defines a compact cobordism between the slice where $R=-1$ (which is a single point) and the slice where $R=R_{0}$. The cobordism is compact because we are below the bubbling threshold for holomorphic disks and spheres. Sending $R_{0}\to \infty$ and appealing to standard Floer theoretic breaking-and-gluing results then proves that $\mathfrak{e}(\zeta)=1$.
\end{proof}

The reader will notice that we have not yet invoked the definition of the boundary depth $\beta(H_{\delta,t})$. In the final step of the argument, we will use it to prove:
\begin{proposition}
  Every representative of the unit element in $\mathrm{CF}_{<E_{+}+\epsilon}$ is sent by $\mathfrak{e}$ to $1$.
\end{proposition}
\begin{proof}
  We recall that the unit element is the cohomology class of $\zeta$, as defined above (by the PSS construction). Thus it suffices to prove that every element of the form $\zeta+d\mu$ which lies in $\mathrm{CF}_{<E_{+}+\epsilon}$ is sent by $\mathfrak{e}$ to $1$.

  Clearly, if $\mathfrak{e}$ was a well-defined chain map on the entire chain complex $\mathrm{CF}$, then $\mathfrak{e}(\zeta+d\mu)=\mathfrak{e}(\zeta)=1$ would hold automatically. However, $\mathfrak{e}$ is not a well-defined chain map on all of $\mathrm{CF}$. However, our earlier discussion establishes that $\mathfrak{e}$ is a chain map on $\mathrm{CF}_{<\beta+E_{+}+2\epsilon}$.

  Since $d\mu\in \mathrm{CF}_{<E_{+}+\epsilon}$ is exact in $\mathrm{CF}$, the definition of the boundary depth implies $d\mu$ is exact in $\mathrm{CF}_{<\beta+E_{+}+2\epsilon}$. Thus $d\mu=d\mu'$ where $\mu'\in \mathrm{CF}_{<\beta+E_{+}+2\epsilon}$, and hence:
  \begin{equation*}
    \mathfrak{e}(\zeta+d\mu)=\mathfrak{e}(\zeta+d\mu')=\mathfrak{e}(\zeta)=1,
  \end{equation*}
  since $\mathfrak{e}$ is a chain map on $\mathrm{CF}_{<\beta+E_{+}+2\epsilon}$. This completes the proof.
\end{proof}

To finish the proof of Lemma \ref{lemma:3}, we recall the definition of the spectral invariant as a homological min-max:
\begin{equation*}
  c([M];H_{\delta,t}):=\inf_{\mu}\set{\text{highest action orbit appearing in }\zeta+d\mu}.
\end{equation*}
Clearly, it suffices to infimize over $\zeta+\d\mu\in \mathrm{CF}_{<E_{+}+\epsilon}$. However, each such chain is sent by $\mathfrak{e}$ to $1\in \Z/2$. In particular, for any such chain $\zeta+\d\mu$, at least one of the capped orbits $(x,v)$ appearing in the chain appears as the left asymptotic of a solution $u$ to \eqref{eq:defn-mathfrak-e} which defines $\mathfrak{e}$. The energy of $u$ is at most $\text{action of }(x,v)-E_{-}$, and hence:
\begin{equation*}
  \text{action of }(x,v)-E_{-}\ge 0\implies c([M];H_{\delta,t})\ge E_{-}.
\end{equation*}
The proof is completed by taking a limit of the perturbed systems $H_{\delta,t}$, as explained above.\hfill$\square$

\subsection{Proof of Theorem \ref{theorem:5}}
\label{sec:proof-theorem-5}

We recall the set-up: $N\subset (M,\alpha)$ is a hypersurface of restricted contact type. The Liouville vector field $Z$ is transverse to $N$, and so a neighborhood of $N$ foliated by level sets of a function $r$ so that:
\begin{enumerate}
\item $r|_{N}=1$,
\item $\d r(Z)=r$.
\end{enumerate}
The Hamiltonian vector field $X_{r}$ has orbits of the form $\rho_{s}(\gamma(t))$ where $\gamma(t)$ is an orbit of $X_{r}$ inside of $N$ and $\rho_{s}$ is the Liouville flow for some time $s$.

Fix some $A$ less than the minimal period of the Reeb flow minus $\epsilon$.

Let $f_{\delta}:\R\to \R$ be a non-negative bump function with $\max f_{\delta}=A$, with support in $(e^{-\delta},e^{\delta})$, and suppose that $\delta$ is very small. We additionally require that $f_{\delta}'(r)$ is a period of the Reeb flow only when:
\begin{enumerate}
\item $f_{\delta}(r)\le \delta$,
\item $f_{\delta}(r)\ge A-\delta$.
\end{enumerate}
Consider $H=f_{\delta}(r)$ as a Hamiltonian function. We will compute the possible actions of contractible orbits. The action of a contractible orbit $\eta(t)$ is:
\begin{equation*}
  a(\eta)=\int H(\eta(t))-\int \eta^{*}\alpha,
\end{equation*}
Each such orbit satisfies $\eta(t)=\rho_{\log(r(\eta))}(\gamma_{\eta}(Tt))$ where $\gamma_{\eta}(t)$ is a $T$-periodic orbit for the Reeb flow in $N$, and $T=f_{\delta}'(r(\eta))$. Notice that the value $r(\eta)$ is constant along the flow.

Case 1: if $f_{\delta}(r)\le \delta$, then there are three possibilities:
\begin{equation*}
  \text{$a(\eta)=0$, $a(\eta)\ge e^{-\delta}T$ or $a(\eta)\le \delta-e^{-\delta}T$,}
\end{equation*}
where $T>0$ is a \emph{positive} period of the Reeb flow.

Case 2: if $f_{\delta}(r)\ge A-\delta$, then there are again three possibilities:
\begin{equation*}
  \text{$a(\eta)=A$, $a(\eta)\ge A-\delta+e^{-\delta}T$, or $a(\eta)\le A-e^{-\delta}T$,}
\end{equation*}
where $T>0$ is a positive period as in Case 1.

Since $A\le \abs{T}-\epsilon$, we can pick $\delta$ small enough so that either the action of $\eta$ is non-positive, or:
\begin{equation*}
  a(\eta)\ge \min\set{A,e^{-\delta}T,A-\delta+e^{-\delta}T}\ge A,
\end{equation*}
where the latter inequality holds as $\delta\to 0$.

Since the spectral invariant $c([M],H)$ is non-negative, the spectral invariant $c([M],-H)$ is non-positive, and the spectral norm is non-degenerate, $c([W],H)$ must be strictly positive. Since $c([M],H)$ is the action of some orbit, we must have that $c([M],H)\ge A$. The desired result then follows by taking the limit $\epsilon\to 0$.\hfill$\square$

\subsection{Proof of Theorem \ref{theorem:7}}
\label{sec:proof-theor-7}

The goal is to show that the measurement $m$ defined in \eqref{eq:measurement} satisfies \ref{M1} and \ref{M2}.

Consider the Floer continuation map $\mathrm{CF}(H_{t}^{k},J)\to \mathrm{CF}(0,J)$ associated to the linear interpolation from $H_{t}^{k}$ to $0$. This map is action decreasing, up to an error,
\begin{equation*}
  \textstyle(\text{action of input})-(\text{action of output})-\int_{0}^{1}\min H_{t}^{k}dt\ge 0.
\end{equation*}
Since the spectral invariant of the unit with respect to the zero system has action equal to $0$, it follows that:
\begin{equation*}
  c([M],H_{t}^{k})\ge \int_{0}^{1}\min H_{t}^{k}dt\ge k\int_{0}^{1}\min H_{t}-\max H_{t}dt,
\end{equation*}
where we use that $H_{t}^{k}$ is mean-normalized to conclude the maximum is non-negative and the minimum is non-positive. Therefore:
\begin{equation*}
  -c([M],H_{t}^{k})\le k(\text{Hofer length of }H_{t}).
\end{equation*}
The spectral invariant is sensitive only to the time-1 map in the universal cover, and hence we can infimize over isotopies to conclude:
\begin{equation*}
  -c([M],H_{t}^{k})\le k\norm{\varphi_{t}}_{\mathrm{Hofer}}.
\end{equation*}
Finally, dividing by $k$ and taking the limit implies $m(\varphi_{t})\le \norm{\varphi_{t}}_{\mathrm{Hofer}}$, which is exactly \ref{M1}.

Next we establish property \ref{M2}, so suppose $H_{t}$ is supported in a displaceable Darboux ball $B$. Then $H_{t}^{k}-k\mu(H_{t})$ is mean-normalized, where $\mu(H_{t})$ is the time-dependent mean of $H_{t}$, and hence:
\begin{equation*}
  m(\varphi_{t})=-\lim_{k\to \infty}\frac{c([M],H_{t}^{k}-k\mu(H_{t}))}{k}.
\end{equation*}
It is well-known property of the Floer homology spectral invariants that:
\begin{equation*}
  c([M],H_{t}^{k}-k\mu(H_{t}))=c([M],H_{t}^{k})-k\int_{0}^{1}\mu(H_{t})dt,
\end{equation*}
i.e., adding a time-dependent constant simply shifts the spectral invariant appropriately. Thus:
\begin{equation*}
  m(\varphi_{t})=\int_{0}^{1}\mu(H_{t})dt-\lim_{k\to \infty}\frac{c([M],H_{t}^{k})}{k}=\int_{0}^{1}\mu(H_{t})dt=\frac{\mathrm{Cal}(\varphi_{1})}{\mathrm{Vol}(M)}.
\end{equation*}
The middle equality holds because the spectral capacity of a displaceable ball is bounded, and the final equality holds by definition of the Calabi invariant. This proves \ref{M2}.\hfill$\square$

\appendix

\section{Spectral invariants from the nonarchimedean perspective}
\label{sec:nonarchimedean}

In this appendix we briefly review the definition of spectral invariants in semiconvex manifolds, and explain how to use the nonarchimedean perspective from \cite{usher-zhang-GT-2016,usher-ASENS-2013} to recover \cite[Theorem 5.1]{entov-polterovich-compositio-2009} .

\subsection{Floer complex over the universal Novikov field}
\label{sec:floer-complex-over}

Let $H_{t}$ be a Hamiltonian system whose time-1 map has non-degenerate fixed points, and define:
\begin{equation*}
  \mathrm{CF}(H_{t};\Lambda):=\set{F:\R\to V(H_{t}):F|_{(-\infty,L)}\text{ has finite support}}.
\end{equation*}
Here $V(H_{t})$ is the free $\Z/2$-vector space generated by the $1$-periodic orbits of $H_{t}$. Addition is defined as the usual addition of function.

In a similar manner, we define the universal Novikov field:
\begin{equation*}
  \Lambda=\set{\lambda:\R\to \Z/2:\lambda|_{(-\infty,L)}\text{ has finite support}}.
\end{equation*}
There are multiplication operations induced by \emph{discrete convolution}; e.g.,
\begin{equation*}
  (\lambda \cdot F)(a)=\sum_{b+c=a}\lambda(b)F(c).
\end{equation*}
This defines a multiplication $\Lambda\otimes \mathrm{CF}(H_{t};\Lambda)\to \mathrm{CF}(H_{t};\Lambda)$. Similar formulas give a multiplication $\Lambda\otimes \Lambda\to \Lambda$ in such a way that $\Lambda$ becomes a field, and then $\mathrm{CF}(H_{t};\Lambda)$ becomes a vector space over $\Lambda$. See, for instance, \cite{hutchings-zeta-24} for this perspective on the Novikov field.

It is convenient to introduce the symbol $\tau^{a}\gamma\in \mathrm{CF}(H_{t};\Lambda)$ to represent the element defined by:
\begin{equation*}
  (\tau^{a}\gamma)(a')=\delta_{a,a'}\gamma,
\end{equation*}
where $\delta_{a,a'}=1$ if $a=a'$ and is zero otherwise, and $\gamma\in V(H_{t})$. Then any element in $\mathrm{CF}(H_{t};\Lambda)$ is a semi-infinite sum of terms of the form $\tau^{a}\gamma$.

It is not hard to see that, if $\gamma_{1},\dots,\gamma_{k}$ is the complete list of orbits of $H_{t}$, and $a_{1},\dots,a_{k}$ are arbitrary numbers, then $\tau^{a_{1}}\gamma_{1},\dots,\tau^{a_{k}}\gamma_{k}$ forms a basis for $\mathrm{CF}(H_{t};\Lambda)$ over the field $\Lambda$.

\subsection{Nonarchimedean filtration}
\label{sec:nonarch-filtr}

On $V(H_{t})$ we define the nonarchimedean filtration:\footnote{This $\ell$ should not be confused with the Lagrangian capacity in \S\ref{sec:lagr-bound-depth}.}
\begin{equation*}
  \ell(\sum c_{i}\gamma_{i})=\sup\set{\int_{0}^{1} H_{t}(\gamma_{i}(t))dt:c_{i}\ne 0}.
\end{equation*}
This extends to the following nonarchimedean filtration on $\mathrm{CF}(H_{t};\Lambda)$ by:
\begin{equation*}
  \ell(F)=\sup\set{\ell(F(a))-a:a\in \R}.
\end{equation*}
One easily verifies for any $\lambda_{1},\dots,\lambda_{k}\in \Lambda$ that:
\begin{equation*}
  \ell\bigg(\sum_{i=1}^{k} \lambda_{i}\tau^{a_{i}}\gamma_{i}\bigg)=\max\set{\ell(\lambda_{1}\tau^{a_{1}}\gamma_{1}),\dots,\ell(\lambda_{k}\tau^{a_{k}}\gamma_{k})}.
\end{equation*}
In other words, $\set{\tau^{a_{1}}\gamma_{1},\dots,\tau^{a_{k}}\gamma_{k}}$ is an \emph{orthogonal basis} for $\mathrm{CF}(H_{t};\Lambda)$. Thus $(\mathrm{CF}(H_{t};\Lambda),\ell)$ is an orthogonalizable $\Lambda$-space, and the results of \cite{usher-zhang-GT-2016} can be applied.

\subsection{The Floer differential}
\label{sec:floer-differential}

In this section we briefly recall the definition of the differential on $\mathrm{CF}(H_{t};\Lambda)$ using the auxiliary data of a generic $\omega$-tame complex structure $J$. In the semipositive framework, one fixes a generic $J$ so that the moduli space of simple $J$-holomorphic curves is transversally cut out. For generic $H_{t}$, the moduli space of finite energy \emph{Floer cylinders} with asymptotics $\gamma_{-},\gamma_{+}$ and symplectic area $b$, i.e., solutions of:
\begin{equation}\label{eq:floer-differential}
  \left\{
    \begin{aligned}
      &u:\R\times \R/\Z\to M,\\
      &\bd_{s}u+J(u)(\bd_{t}u-X_{t}(u))=0,\\
      &\textstyle\lim_{s\to\pm \infty}u(s,t)=\gamma_{\pm}(t)\text{ and }\omega(u)=b
    \end{aligned}
  \right.
\end{equation}
will be compact up to translations and the usual breaking of Floer trajectories. In the semipositivity framework, it is important that $H_{t}$ is generic in order for this compactness to hold, because one obstructs the bubbling of $J$-holomorphic spheres using general position arguments.

Let us denote by $I(\gamma_{-},\gamma_{+},b)$ the finite count (reduced modulo 2) of rigid-up-to-translation solutions to \eqref{eq:floer-differential}. Define on $\mathrm{CF}(H_{t};\Lambda)$ the map:
\begin{equation*}
  d(\tau^{a}\gamma_{-})=\sum_{b,\gamma_{+}}I(\gamma_{-},\gamma_{+},b)\tau^{a+b}\gamma_{+}.
\end{equation*}
In words, the differential counts all the rigid-up-to-translation Floer cylinders, using the Novikov coefficients to keep track of their symplectic areas. It is well-known that $d^{2}=0$ provided $H_{t}$ is generic.

\subsection{Floer complex of capped orbits}
\label{sec:comp-with-floer}

Let $\mathrm{CF}(H_{t})$ denote the vector space of semi-infinite sums of capped orbits $(\gamma,v)$. Here the capping $v$ is an equivalence class of disks bounding $\gamma$; two disks are equivalent if their difference forms a sphere with zero symplectic area and zero first Chern number. The sums are semi-infinite in that for any $L\in \R$, a sum can have only finitely many terms $(\gamma,v)$ with $\omega(v)<L$.

Define a morphism: $\iota:\mathrm{CF}(H_{t})\to \mathrm{CF}(H_{t};\Lambda)$ by:
\begin{equation*}
  \iota:(\gamma,v)\mapsto \tau^{\omega(v)}\gamma.
\end{equation*}
It is straightforward to check that this morphism is well-defined.

We obtain a nonarchimedean filtration on $\mathrm{CF}(H_{t})$ by pulling back the nonarchimedean filtration on $\mathrm{CF}(H_{t};\Lambda)$. Moreover, one can define the Floer differential on $\mathrm{CF}(H_{t})$ (using the same equation as \eqref{eq:floer-differential}) in such a way that $\iota$ becomes a chain map.

Pick an auxiliary section $\mathfrak{s}$ of the determinant line bundle of $(M,J)$ whose zero set is disjoint from all orbits of $H_{t}$. The signed intersection number between a capping and $\mathfrak{s}^{-1}(0)$ is well-defined (independent of the representative). The section can also be used to define Conley-Zehnder indices $\mathrm{CZ}_{\mathfrak{s}}(\gamma)$ in such a way that:
\begin{equation}\label{eq:degree}
  \deg(\gamma,v)=n-\mathrm{CZ}_{\mathfrak{s}}(\gamma)-2\mathfrak{s}^{-1}(0)\cdot v
\end{equation}
defines a grading on $\mathrm{CF}(H_{t})$ so that the Floer differential decreases grading by $1$. The normalization of \eqref{eq:degree} is chosen so that the PSS morphism sends a cycle of dimension $k$ to a Floer cycle of grading $k$.

We then have the following extension of coefficients lemma (see \cite[Proposition 2.4]{mak-sun-varolgunes-jtopol-2024} for the same result):
\begin{lemma}
  Let $\zeta\in \mathrm{CF}_{k}(H_{t})$ lie in the $k$th graded piece. Then:
  \begin{equation*}
    \inf\set{\ell(\zeta+d \eta):\eta\in \mathrm{CF}_{k-1}(H_{t})}=\inf\set{\ell(\iota(\zeta)+d \beta):\beta\in \mathrm{CF}(H_{t};\Lambda)}.
  \end{equation*}
  In other words, extending coefficients to the universal Novikov field does not decrease the nonarchimedean distance to the subspace of exact elements.
\end{lemma}
\begin{proof}
  It suffices to prove the $\ge$ inequality. Let $\Pi_{k}\subset \mathrm{CF}(H_{t};\Lambda)$ be the subspace over $\Z/2$ of semi-infinite sums spanned by terms $\tau^{a}\gamma$ satisfying:

  \begin{enumerate}[label=$\bullet$]
  \item $a$ is \emph{not} the symplectic area of capping $v$ so $\deg(\gamma,v)=k$.
  \end{enumerate}

  We claim that $d\Pi_{k-1}\subset\Pi_{k}$. Indeed, if $\tau^{a+b}\gamma_{+}$ appears in the output of $d(\tau^{a}\gamma_{-})$ then the there is an index $1$ Floer cylinder of area $b$ joining $\gamma_{-}$ to $\gamma_{+}$; this observation proves the claim.

  Then any element $\beta$ can be decomposed as $\beta=\iota(\eta)+\pi$ where $\pi\in \Pi_{k-1}$. We compute:
  \begin{equation*}
    \ell(\iota(\zeta)+d \beta)=\ell(\iota(\zeta+d\eta)+d\pi)\ge \ell(\iota(\zeta+\d\eta))=:\ell(\zeta+d\eta).
  \end{equation*}
  The $\ge$ inequality follows since $\iota(\zeta+\d\eta)$ and $\Pi_{k}$ are \emph{orthogonal}. Indeed, let us consider any term $\tau^{a}\gamma$ which optimizes $\ell(\zeta+d\eta)$. This term cannot be cancelled by any term appearing in $d\pi$, by definition. Hence $\tau^{a}\gamma$ still appears in $\iota(\zeta+\d\eta)+\d\pi$ with non-zero coefficient, and the desired inequality follows. This completes the proof.
\end{proof}

\subsection{Spectral invariant of the unit}
\label{sec:pss-spectr-invar}

First suppose that $M$ is compact. Recall from Proposition \ref{prop:PSS-e} the cycle $\zeta\in \mathrm{CF}(H_{t})$ obtained by counting rigid PSS cylinders. This cycle lies in the $2n$ graded piece of $\mathrm{CF}(H_{t})$. Then:
\begin{equation*}
  c([M],H_{t})=\inf\set{\ell(\iota(\zeta)+d\beta):\beta\in \mathrm{CF}(H_{t};\Lambda)}
\end{equation*}
is an invariant of the (generic) system $H_{t}$ and the complex structure $J$.

Some care is needed when $M$ is open, in which case $M$ is assumed to be semi-convex (has a non-compact end modelled on $S_{+}Y\times T$). In this case, we require that $H_{t}$ has a \emph{split equivariant negative ideal restriction}. This means that, in the non-compact end, the system generated by $H_{t}$ is:
\begin{enumerate}
\item \emph{split}, i.e., of the form $(\varphi_{t},\phi_{t})$ where $\varphi_{t},\phi_{t}$ are Hamiltonian systems on $S_{+}Y$ and $T$, respectively,
\item \emph{equivariant}, i.e., $\varphi_{t}$ is equivariant under the Liouville flow on $S_{+}Y$,
\item \emph{negative}, i.e., the system $\varphi_{t}$ is generated by an asymptotically negative Hamiltonian function.
\end{enumerate}

Assuming equivariance, negativity is equivalent to requiring that the ideal restriction of $\varphi_{t}$ (a contact isotopy on $Y$) is a negative path in the contactomorphism group.

In this case, one can prove an energy estimate for solutions to the PSS equation (this is not the case if $\varphi_{t}$ has a positive ideal restriction). Continuing to assume that the system generated by $H_{t}$ has non-degenerate $1$-periodic orbits, one obtains a maximum principle and the compactness up-to-breaking of the moduli space of solutions to the PSS equation, following, e.g., \cite{brocic-cant-JFPTA-2024,alizadeh-atallah-cant-arXiv-2023}. The upshot of this discussion is that the PSS cycle $\zeta$ and the spectral invariant $c([M],H_{t})$ can be defined as in the compact case.

Finally, one extends the definition of $c([M],H_{t})$ to systems where $\varphi_{t}$ is \emph{non-positive} by a limiting process, approximating $\varphi_{t}$ by negative systems. The details in the case when $T=\mathrm{pt}$ and $\varphi_{t}=\id$ can be found in \cite{alizadeh-atallah-cant-arXiv-2023} and the general case follows from the same argument.

\subsection{A product formula for spectral invariants}
\label{sec:prod-form-spectr}

Let $M$ be strongly semipositive and semiconvex. As above, let the non-compact end of $M$ be modelled on $S_{+}Y\times T$. Then $M\times T^{2}$ is semipositive and semiconvex; the non-compact end of $M$ is modelled on $S_{+}Y\times (T\times T^{2})$.

Let $H_{t}$ be a system on $M$ whose flow is split, equivariant, and non-positive, so that $c([M],H_{t})$ is defined by the above procedure. Let $K_{t}$ be any Hamiltonian system on $T^{2}$. Then: $$H_{t}\circ \pi_{1}+K_{t}\circ \pi_{2}$$ is split, equivariant, and non-positive on $M\times T^{2}$ with respect to the non-compact end $S_{+}Y\times T\times T^{2}$. The goal in this section is to prove:
\begin{theorem}[see {\cite[Theorem 5.1]{entov-polterovich-compositio-2009}}]
  The spectral invariants satisfy:
  \begin{equation*}
    c([M\times T^{2}],H_{t}\circ \pi_{1}+K_{t}\circ \pi_{2})=c([M],H_{t})+c([T^{2}],K_{t}),
  \end{equation*}
  where $H_{t},K_{t}$ are as above.
\end{theorem}
The proof is based on the nonarchimedean singular value decomposition result from \cite{usher-zhang-GT-2016} and the analysis of orthogonal bases and tensor products from \cite[\S8]{usher-ASENS-2013}. The result is not new (the non-compact setting does not change things in a significant way) and follows from \cite[Theorem 5.1]{entov-polterovich-compositio-2009}; we include the argument only for completeness.

\begin{proof}
  By continuity, it suffices to prove the case when $H_{t}$ is split, equivariant, and negative (rather than non-positive). We may also assume that $H_{t}$ is generic on the compact part so that all orbits are non-degenerate and the Floer differential is well-defined (for some generic admissible complex structure $J$). We similarly pick $K_{t}$ generically so that the Floer complex is well-defined (using the standard almost complex structure on $T^{2}$).

  Abbreviate $G_{t}=H_{t}\circ \pi_{1}+K_{t}\circ \pi_{2}$, and observe that the system generated by $G_{t}$ is split with respect to the decomposition $M\times T^{2}$. In particular, every orbit of $G_{t}$ is of the form $(\gamma(t),\mu(t))$ where $\gamma,\mu$ are orbits of $H_{t},K_{t}$, respectively. This induces an isomorphism:
  \begin{equation}\label{eq:tensor-product}
    \mathrm{CF}(H_{t};\Lambda)\otimes_{\Lambda} \mathrm{CF}(K_{t};\Lambda)\to \mathrm{CF}(G_{t};\Lambda),
  \end{equation}
  satisfying $(\tau^{a}\gamma)\otimes (\tau^{b}\mu)\mapsto \tau^{a+b}(\gamma,\mu)$.

  The key analytic input is that \eqref{eq:tensor-product} is a chain map, provided one uses the split complex structure on $M\times T^{2}$, and where the differential on the left hand side of \eqref{eq:tensor-product} is $d\otimes 1+1\otimes d$. This is fairly obvious, and has been observed before in, e.g., \cite[\S5.4]{entov-polterovich-compositio-2009}. Moreover, if $\zeta(H_{t})$ and $\zeta(K_{t})$ be the PSS cycles representing the unit element, then the image of $\zeta(H_{t})\otimes \zeta(K_{t})$ under \eqref{eq:tensor-product} is the PSS cycle representing the unit in $\mathrm{CF}(G_{t};\Lambda)$.

  The rest of the proof is entirely algebraic. The first step is to appeal to the nonarchimedean singular value decomposition of \cite{usher-zhang-GT-2016}. This yields an orthogonal $\Lambda$-basis:
  \begin{equation*}
    \set{Z_{1},\dots,Z_{N},T_{1},\dots,T_{M},S_{1},\dots,S_{M}}
  \end{equation*}
  for $\mathrm{CF}(H_{t};\Lambda)$ such that:
  \begin{enumerate}
  \item $T_{1},\dots,T_{M}$ is a basis for $\mathrm{im}(d)$,
  \item $Z_{1},\dots,Z_{N}$ projects to a basis for $\ker(d)/\mathrm{im}(d)$,
  \item $dS_{i}=T_{i}$.
  \end{enumerate}
  One similarly concludes a basis $\set{Z_{1}',\dots,T_{1}',\dots,S_{1}',\dots}$ for $\mathrm{CF}(K_{t};\Lambda)$.

  Write:
  \begin{equation*}
    \zeta(H_{t})=\xi(H_{t})+d\beta(H_{t})\text{ and }\zeta(K_{t})=\xi(K_{t})+d\beta(K_{t}),
  \end{equation*}
  where $\xi(H_{t}),\xi(K_{t})$ lie in the span of the $Z,Z'$ vectors respectively. It follows from orthogonality of the bases that:
  \begin{equation}\label{eq:carriers}
    c([M],H_{t})=\ell(\xi(H_{t}))\text{ and }c([T^{2}],K_{t})=\ell(\xi(K_{t}));
  \end{equation}
  in other words, the $\xi$ cycles are \emph{carriers} of the spectral invariants.

  The next step is to appeal to the results on tensor products of orthogonalizable $\Lambda$-vector spaces from \cite[\S8]{usher-ASENS-2013}. There it is proven (in a more general context) that:
  \begin{equation}\label{eq:formula-tp}
    \ell(\textstyle\sum_{i,j} \lambda_{i}\lambda_{j}'Z_{i}\otimes Z_{j}')=\max_{i,j}\set{\ell(\lambda_{i}Z_{i})+\ell(\lambda_{j}Z_{j}')},
  \end{equation}
  where we abuse notation and let $Z_{i}\otimes Z_{j}$ denote its image under \eqref{eq:tensor-product}. It is also shown that the tensor product of orthogonal bases is sent to an orthogonal basis.

  Thus we conclude that:
  \begin{equation*}
    \zeta(H_{t})\otimes \zeta(K_{t})=\xi(H_{t})\otimes \xi(K_{t})+d\beta.
  \end{equation*}
  We then conclude:
  \begin{equation*}
    \begin{aligned}
      c([M\times T^{2}],G_{t})
      &=\inf\set{\ell(\zeta(H_{t})\otimes \zeta(K_{t})+d\beta):\beta\in \mathrm{CF}(G_{t};\Lambda)}\\
      &=\inf\set{\ell(\xi(H_{t})\otimes \xi(K_{t})+d\beta):\beta\in \mathrm{CF}(G_{t};\Lambda)}\\
      &=\ell(\xi(H_{t})\otimes\xi(K_{t}))\\
      &=\ell(\xi(H_{t}))+\ell(\xi(K_{t}))\\
      &=c([M],H_{t})+c([T^{2}],K_{t}),
    \end{aligned}
  \end{equation*}
  where we have used the orthogonality of the tensor product basis in the third equality, \eqref{eq:formula-tp} in the fourth equality, and \eqref{eq:carriers} in the final equality. This completes the proof.
\end{proof}

\subsection{On the rationality assumption in continuity arguments}
\label{sec:rati-assumpt-cont}

In the proof of Theorem \ref{theorem:1} (in particular, Lemma \ref{lemma:technical-estimate}) we appealed to rationality to conclude that spectral invariants are constant along a deformation provided the contractible orbits and their actions are constant. The reasoning is that the spectral invariants are continuous and valued inside the spectrum of the system. However, this is only known to hold in general when the symplectic manifold is rational, or the systems remain non-degenerate.

The following shows that one can drop the rationality assumption in the proof of Lemma \ref{lemma:technical-estimate}.
\begin{lemma}\label{lemma:irrational-stability}
  Let $H_{s,t}$ be a family of Hamiltonian functions which are split, equivariant, and negative in the non-compact end of a semiconvex and semipositive symplectic manifold $M$. Suppose that the set of contractible orbits of $H_{s,t}$ is independent of $s$, and the values of $H_{s,t}$ in a neighborhood containing these orbits are also independent of $s$. Then it holds that:
  \begin{equation*}
    c([M],H_{0,t})=c([M],H_{1,t}).
  \end{equation*}
\end{lemma}
\begin{proof}
  The shortest argument is to perturb $H_{s,t}$ \emph{in an $s$-independent way} in a neighborhood of the contractible orbits, so that each contractible orbit of $H_{s,t}$ becomes non-degenerate. Then one can appeal to the non-degenerate spectrality proved in \cite{usher-compositio-2008}. This completes the proof.

  An alternative argument is decompose the continuation: $$\mathrm{HF}(H_{0,t})\to \mathrm{HF}(H_{1,t})$$ into a composition of many continuation maps, and prove (under the hypotheses of the lemma) that such a continuation map preserves the action of any cycle made of contractible orbits. We leave the details of such an approach to the reader.
\end{proof}

\bibliographystyle{alpha}
\bibliography{citations}

\begin{thebibliography}{MSV24}

\bibitem[AA23]{allais-arlove-arXiv-2023}
S.~Allais and P-A. Arlove.
\newblock Spectral selectors and contact orderability.
\newblock arXiv:2309.10578, 2023.

\bibitem[AAC23]{alizadeh-atallah-cant-arXiv-2023}
H.~Alizadeh, M.~S. Atallah, and D.~Cant.
\newblock Lagrangian intersections and the spectral norm in convex-at-infinity
  symplectic manifolds.
\newblock arXiv:2312.14752, December 2023.

\bibitem[AAC24]{alizadeh-atallah-cant-arXiv-2024}
H.~Alizadeh, M.~S. Atallah, and D.~Cant.
\newblock The spectral diameter of a symplectic ellipsoid.
\newblock arXiv:2408.07214, August 2024.

\bibitem[AM18]{albers-merry-JSG-2018}
P.~Albers and W.~J. Merry.
\newblock Orderability, contact non-squeezing, and {R}abinowitz {F}loer
  homology.
\newblock {\em J. Symp. Geom.}, 16(6):1481--1547, 2018.

\bibitem[Aud88]{audin-CMH-1988}
M.~Audin.
\newblock Fibrés normaux d'immersions en dimension double, point doubles
  d'immersions lagrangiennes et plongements totalement réels.
\newblock {\em Comment. Math. Helv.}, 63:593--623, 1988.

\bibitem[BC24]{brocic-cant-JFPTA-2024}
F.~Bro\'ci\'c and D.~Cant.
\newblock Bordism classes of loops and {F}loer's equation in cotangent bundles.
\newblock {\em J. Fixed Point Theory Appl.}, 26:1--29, 2024.

\bibitem[BIP08]{burago-ivanov-polterovich-2008}
D.~Burago, S.~Ivanov, and L.~Polterovich.
\newblock Conjugation-invariant norms on groups of geometric origin.
\newblock In {\em Groups of Diffeomorphisms}, volume~52 of {\em Adv. Stud. Pure
  Math.}, pages 221--250. Math. Soc. Japan, 2008.

\bibitem[BK17]{brandenbursky-kedra-ann-math-que-2017}
M.~Brandenbursky and J.~Kedra.
\newblock The autonomous norm on {$\mathrm{Ham}(\mathbf{R}^{2n})$ is bounded}.
\newblock {\em Ann. Math. Qu\'e.}, 41(1):63--65, 2017.

\bibitem[Bol96]{bolle-CRAS-1996}
P.~Bolle.
\newblock Une condition de contact pour les sous-variétés coïsotropes d'une
  variété symplectique.
\newblock {\em Comptes Rendus Acad. Sci.}, 1(322):83--86, 1996.

\bibitem[Bol98]{bolle-math-Z-1998}
P.~Bolle.
\newblock A contact condition for {$p$}-dimensional submanifolds of a
  symplectic manifold $(2\le p\le n)$.
\newblock {\em Math. Z.}, 227:211--230, 1998.

\bibitem[Bor12]{borman-JSG-2012}
M.~S. Borman.
\newblock Symplectic reduction of quasi-morphisms and quasi-states.
\newblock {\em J. Symplectic Geom.}, 10(2):225--246, 2012.

\bibitem[Can23]{cant-arXiv-2023}
D.~Cant.
\newblock {S}helukhin's {H}ofer distance and a symplectic cohomology barcode
  for contactomorphisms.
\newblock arXiv:2309.00529, September 2023.

\bibitem[CGP24]{c-gardiner-prasad-arXiv-2024}
D.~Cristofaro-Gardiner and R.~Prasad.
\newblock Low-action holomorphic curves and invariants sets.
\newblock arXiv:2401.14445, 2024.

\bibitem[Cha15]{charette-AGT-2015}
F.~Charette.
\newblock Gromov width and uniruling for orientable {L}agrangian surfaces.
\newblock {\em Alg. Geom. Topol.}, 15:1439--1451, 2015.

\bibitem[Che96]{chekanov-math-z-1996}
Y.~V. Chekanov.
\newblock Lagrangian tori in a symplectic vector space and global
  symplectomorphisms.
\newblock {\em Math. Z.}, 223:547--559, 1996.

\bibitem[Che98]{chekanov-duke-1998}
Y.~V. Chekanov.
\newblock {L}agrangian intersections, symplectic energy, and areas of
  holomorphic curves.
\newblock {\em Duke Math. J.}, 95(1), 1998.

\bibitem[CM18]{cieliebak-mohnke-inventiones-2018}
K.~Cieliebak and K.~Mohnke.
\newblock Punctured holomorphic curves and lagrangian embeddings.
\newblock {\em Invent. math.}, 212:213--295, 2018.

\bibitem[C{\^o}t20]{cote-JSG-2020}
L~C{\^o}t{\'e}.
\newblock On linking of {L}agrangian tori in {$\R^{4}$}.
\newblock {\em J. Symplectici Geom.}, 18(2):409--462, 2020.

\bibitem[CS16]{casals-spacil-j-topol-anal-2016}
R.~Casals and Spá{\v{c}}il.
\newblock Chern-weil theory and the group of strict contactomorphisms.
\newblock {\em J. Topol. Anal.}, 8(1):59--87, 2016.

\bibitem[CU24]{cant-uljarevic-arxiv-2024}
D.~Cant and I.~Uljarevi\'c.
\newblock Selective {F}loer cohomology for contact vector fields.
\newblock arXiv:2405.05443, May 2024.

\bibitem[Dra08]{dragnev-CPAM-2008}
D.~Dragnev.
\newblock Symplectic rigidity, symplectic fixed points, and global
  perturbations of {H}amiltonian systems.
\newblock {\em Comm. Pure Appl. Math.}, 61(3):346--370, 2008.

\bibitem[DUZ23]{djordjevic-uljarevic-zhang-arXiv-2023}
D.~Djordjevi{\'c}, I.~Uljarevi{\'c}, and J.~Zhang.
\newblock Quantitative characterization in contact {H}amiltonian dynamics -
  {I}.
\newblock arXiv:2309.00527, 2023.

\bibitem[EH90]{ekeland-hofer-math-z-1990}
I.~Ekeland and H.~Hofer.
\newblock Symplectic topology and {H}amiltonian dynamics {II}.
\newblock {\em Math. Z.}, 203:553--567, 1990.

\bibitem[EKP06]{eliashberg-kim-polterovich-2006}
Y.~Eliashberg, S.~S. Kim, and L.~Polterovich.
\newblock Geometry of contact transformations and domains: orderability versus
  squeezing.
\newblock {\em Geom. Topol.}, 10(3):1635--1747, 2006.

\bibitem[Eli91]{eliashberg-JAMS-1991}
Y.~Eliashberg.
\newblock New invariants of open symplectic and contact manifolds.
\newblock {\em J. Amer. Math. Soc.}, 4(3), 1991.

\bibitem[EP03]{entov-poltero-IMRN-2003}
M.~Entov and L.~Polterovich.
\newblock Calabi quasimorphism and quantum homology.
\newblock {\em Int. Math. Res. Not.}, 2003:1635--1676, 2003.

\bibitem[EP09]{entov-polterovich-compositio-2009}
M.~Entov and L.~Polterovich.
\newblock Rigid subsets of symplectic manifolds.
\newblock {\em Compositio Math.}, 145(3):773--826, 2009.

\bibitem[FGS05]{frauenfelder-ginzburg-schlenk}
U.~Frauenfelder, V.~L. Ginzburg, and F.~Schlenk.
\newblock Energy capacity inequalities via an action selector.
\newblock In {\em Geometry, spectral theory, groups, and dynamics}, volume 387
  of {\em Contemp. Math.}, pages 129--152. Amer. Math. Soc., 2005.

\bibitem[FH23]{fish-hofer-annals-2023}
J.~W. Fish and H.~Hofer.
\newblock Feral curves and minimal sets.
\newblock {\em Ann. Math.}, 197(2):533--738, 2023.

\bibitem[FS07]{frauenfelder-schlenk-IJM-2007}
U.~Frauenfelder and F.~Schlenk.
\newblock Hamiltonian dynamics on convex symplectic manifolds.
\newblock {\em Isr. J. Math.}, 159:1--56, 2007.

\bibitem[FZ24]{feng-zhang-arXiv-2024}
Q.~Feng and J.~Zhang.
\newblock Spectrally-large scale geometry in cotangent bundles.
\newblock {\em arXiv:2401.17590}, 2024.

\bibitem[GG15]{ginzburg-gurel-math-Z-2015}
V.~L. Ginzburg and B.~Z. Gürel.
\newblock Fragility and persistence of leafwise intersections.
\newblock {\em Math. Z.}, 280:989--1004, 2015.

\bibitem[Gin05]{ginzburg-2005-weinstein}
V.~L. Ginzburg.
\newblock The {W}einstein conjecture and theorems of nearby and almost
  existence.
\newblock In {\em The Breadth of Symplectic and Poisson Geometry: Festschrift
  in Honor of Alan Weinstein}, Progr. Math., pages 139--172. Birkh{\"a}user,
  2005.

\bibitem[Gin07]{ginzburg-duke-2007}
V.~L. Ginzburg.
\newblock Coisotropic intersections.
\newblock {\em Duke Math. J.}, 140(1):111--163, 2007.

\bibitem[Giv90]{givental-advsov-1990}
A.B. Givental.
\newblock Nonlinear generalization of the {M}aslov index.
\newblock In {\em Theory of singularities and its applications}, volume~1 of
  {\em Adv. Sov. Math.}, pages 71--103. AMS, 1990.

\bibitem[GPR22]{gutt-pereira-ramos-arXiv-2022}
J.~Gutt, M.~Pereira, and V.~G.~B. Ramos.
\newblock Cube normalized symplectic capacities.
\newblock arXiv:2208.13666, 2022.

\bibitem[Gro86]{gromov-springer-1986}
M.~Gromov.
\newblock {\em Partial Differential Relations}, volume~3 of {\em Ergebnisse der
  Mathematik und ihrer Grenzgebiete}.
\newblock Springer-Verlag, 1986.

\bibitem[GT23]{ganor-tanny-barricades}
Y.~Ganor and S.~Tanny.
\newblock {Floer} theory of disjointly supported {Hamiltonians} on
  symplectically aspherical manifolds.
\newblock {\em Alg. Geom. Topol.}, 23(2):645--732, 2023.

\bibitem[G{\"u}r08]{gurel-CCM-2008}
B.~Z. G{\"u}rel.
\newblock Totally non-coisotropic displacement and its applications to
  {H}amiltonian dynamics.
\newblock {\em Commun. Contemp. Math.}, 10(6):1103--1128, 2008.

\bibitem[G{\"u}r10]{gurel-IMRN-2010}
B.~Z. G{\"u}rel.
\newblock Leafwise coisotropic intersections.
\newblock {\em Int. Math. Res. Not.}, 2010(5):914--931, 2010.

\bibitem[Hin24]{hind-arXiv-2024}
R.~Hind.
\newblock On the {G}romov width of complements of {L}agrangian tori.
\newblock arXiv:2405.03866, May 2024.

\bibitem[HO20]{hind-opshtein-CMH-2020}
R.~Hind and E.~Opshtein.
\newblock Squeezing {L}agrangian tori in dimension 4.
\newblock {\em Comment. Math. Helv.}, 95(3):535--567, 2020.

\bibitem[Hof90]{hofer-PRS-edinburgh-1990}
H.~Hofer.
\newblock On the topological properties of symplectic maps.
\newblock {\em Proc. Roy. Soc. Edinburgh Sect. A: Math.}, 115:25--38, 1990.

\bibitem[HS95]{hofer-salamon-95}
H.~Hofer and D.~Salamon.
\newblock Floer homology and {N}ovikov rings.
\newblock In {\em The Floer memorial volume}, volume 133 of {\em Progr. Math.},
  pages 483--524. Birkh{\"a}user, 1995.

\bibitem[Hut24]{hutchings-zeta-24}
M.~Hutchings.
\newblock Zeta functions of dynamically tame {Liouville domains}.
\newblock arXiv:2402.07003, 2024.

\bibitem[HZ94]{hofer-zehnder-book-1994}
H.~Hofer and E.~Zehnder.
\newblock {\em Symplectic Invariants and Hamiltonian Dynamics}.
\newblock Birkh\"auhser Advanced Texts, 1994.

\bibitem[Kan13]{kang-IMRN-2013}
J.~Kang.
\newblock Rabinowitz {F}loer homology and coisotropic intersections.
\newblock {\em Int. Math. Res. Not.}, 2013(10):2271--2322, 2013.

\bibitem[Kat73]{katok-izv-akad-nauk-1973}
A.~B. Katok.
\newblock Ergodic perturbations of degenerate integrable {H}amiltonian systems.
\newblock {\em Izv. Akad. Nauk. SSSR Ser. Mat.}, 37:539--576, 1973.

\bibitem[Ker08]{kerman-j-mod-dyn-2008}
E.~Kerman.
\newblock Displacement energy of coisotropic submanifolds and {H}ofer's
  geometry.
\newblock {\em J. Mod. Dyn.}, 2(3):471--497, 2008.

\bibitem[KS21]{kislev-shelukhin-GT-2021}
A.~Kislev and E.~Shelukhin.
\newblock Bounds on spectral norms and barcodes.
\newblock {\em Geom. Topol.}, 25:3257--3350, 2021.

\bibitem[KZ22]{khanevsky-ziltener-dga-2022}
M.~Khanevsky and F.~Ziltener.
\newblock A relative {H}ofer estimate and the asymptotic {H}ofer-{L}ipshitz
  constant.
\newblock {\em Diff. Geom. Appl.}, 85:1--26, 2022.

\bibitem[LR20]{lisi-rieser-jsg-2020}
S.~Lisi and A.~Rieser.
\newblock Coisotropic {H}ofer-{Z}ehnder capacities and non-squeezing for
  relative embeddings.
\newblock {\em J. Symplectic Geom.}, 18(3):819--865, 2020.

\bibitem[LS94]{laudenbach-sikorav-IMRN-1994}
F.~Laudenbach and J.-C. Sikorav.
\newblock Hamiltonian disjunction and limits of {L}agrangian submanifolds.
\newblock {\em Int. Math. Res. Not.}, 1994(4):161--168, 1994.

\bibitem[LZ18]{leclercq-zapolsky-JTA-2018}
R.~Leclercq and F.~Zapolsky.
\newblock Spectral invariants for monotone lagrangians.
\newblock {\em J. Topol. Anal.}, 10(3):627--700, 2018.

\bibitem[McD10]{mcduff-CMH-2010}
D.~McDuff.
\newblock Monodromy in {H}amiltonian {F}loer theory.
\newblock {\em Comment. Math. Helv.}, 85(1):95--133, 2010.

\bibitem[Mos78]{moser-acta-1978}
J.~Moser.
\newblock A fixed point theorem in symplectic geometry.
\newblock {\em Acta Math.}, 141:17--34, 1978.

\bibitem[MS12]{mcduff-salamon-book-2012}
D.~McDuff and D.~Salamon.
\newblock {\em $J$-holomorphic curves and Symplectic Topology}.
\newblock American Mathematical Society, Colloquium Publications, 2nd edition,
  2012.

\bibitem[MSV24]{mak-sun-varolgunes-jtopol-2024}
C.~Y. Mak, Y.~Sun, and U.~Varolgunes.
\newblock A characterization of heaviness in terms of relative symplectic
  cohomology.
\newblock {\em J. Topol.}, 17(1):1--26, 2024.

\bibitem[Oh05]{oh-2005-duke}
Y-G. Oh.
\newblock Spectral invariants, analysis of the {F}loer moduli space, and
  geometry of the {H}amiltonian diffeomorphism group.
\newblock {\em Duke Math. J.}, 126(1):199--295, 2005.

\bibitem[Oh09]{oh-JKMS-2009}
Y.-G. Oh.
\newblock Floer mini-max theory, the {C}erf diagram, and the spectral
  invariants.
\newblock {\em J. Korean Math. Soc.}, 46:363--447, 2009.

\bibitem[Ost03]{ostrover-CCM-2003}
Y.~Ostrover.
\newblock A comparison of {H}ofer's metric on {H}amiltonian diffeomorphisms and
  {L}agrangian submanifolds.
\newblock {\em Commun. Contemp. Math.}, 5:2123--2141, 2003.

\bibitem[Per22]{pereira-arXiv-2022}
M.~Pereira.
\newblock On the {L}agrangian capacity of convex or concave toric domains.
\newblock arXiv:2207.11022, 2022.

\bibitem[Pol91]{polterovich-math-Z-1991}
L.~Polterovich.
\newblock Monotone {L}agrange submanifolds of linear spaces and the {M}aslov
  class in cotangent bundles.
\newblock {\em Math. Z.}, 207(2):217--222, 1991.

\bibitem[Pol98]{polterovich-IMRN-1998}
L.~Polterovich.
\newblock {H}ofer's diameter and {L}agrangian intersections.
\newblock {\em Int. Math. Res. Not.}, 1998(4):1--7, 1998.

\bibitem[Pol01]{polterovich-book-2001}
L.~Polterovich.
\newblock {\em The geometry of the group of {H}amiltonian symplectic
  diffeomorphisms}.
\newblock Lectures in Mathematics, ETH Z{\"u}rich. Birkh{\"a}user Basel, 2001.

\bibitem[PR14]{polterovich-rosen-book-2014}
L.~Polterovich and D.~Rosen.
\newblock {\em Function Theory on Symplectic Manifolds}, volume~34 of {\em CRM
  Monograph Series}.
\newblock American Mathematical Society, 2014.

\bibitem[PS23]{polterovich-shelukhin-compositio-2023}
L.~Polterovich and E.~Shelukhin.
\newblock {L}agrangian configurations and {H}amiltonian maps.
\newblock {\em Compositio Math.}, 159:2483--2520, 2023.

\bibitem[PSS96]{piunikhin-salamon-schwarz-1996}
S.~Piunikhin, D.~Salamon, and M.~Schwarz.
\newblock Symplectic {F}loer-{D}onaldson theory and quantum cohomology.
\newblock In {\em Contact and symplectic geometry ({C}ambridge, 1994)},
  volume~8 of {\em Publ. Newton Inst.}, pages 171--200. Cambridge Univ. Press,
  Cambridge, 1996.

\bibitem[San11]{sandon-ann-inst-four-2011}
S.~Sandon.
\newblock Contact homology, capacity and non-squeezing in {$\R^{2n}\times S^1$}
  via generating functions.
\newblock {\em Ann. Inst. Fourier}, 61(1):145--185, 2011.

\bibitem[Sch00]{schwarz-pacific-j-math-2000}
M.~Schwarz.
\newblock On the action spectrum for closed symplectically aspherical
  manifolds.
\newblock {\em Pacific J. Math.}, 193:419--461, 2000.

\bibitem[Sei97]{seidel-GAFA-1997}
P.~Seidel.
\newblock $\pi_1$ of symplectic automorphism groups and invertibles in quantum
  homology rings.
\newblock {\em Geom. funct. anal.}, 7:1046--1095, 1997.

\bibitem[Sik90]{sikorav-pisa-1990}
J-C. Sikorav.
\newblock Systèmes {H}amiltoniens et topologie symplectique.
\newblock ETS Editrice Pisa, 1990.
\newblock
  \url{https://perso.ens-lyon.fr/jean-claude.sikorav/textes/Pise1990.pdf}
  Accessed June 12 2024.

\bibitem[Sul76]{sullivan-inventiones-1976}
D.~Sullivan.
\newblock Cycles for the dynamical study of foliated manifolds and complex
  manifolds.
\newblock {\em Invent. Math.}, 36:225--255, 1976.

\bibitem[SZ13]{swoboda-ziltener-JSG-2013}
J.~Swoboda and F.~Ziltener.
\newblock A symplectically non-squeezable small set and the regular coisotropic
  capacity.
\newblock {\em J. Symplectic Geom.}, 11(4):509--523, 2013.

\bibitem[Ush08]{usher-compositio-2008}
M.~Usher.
\newblock Spectral numbers in {F}loer theories.
\newblock {\em Compositio Math.}, 144:1581--1592, 2008.

\bibitem[Ush10]{usher-CCM-2010}
M.~Usher.
\newblock The sharp energy-capacity inequality.
\newblock {\em Comm. Cont. Math}, 12(3):457--473, 2010.

\bibitem[Ush11]{usher-israel-j-math-2011}
M.~Usher.
\newblock Boundary depth in {F}loer theory and its applications to
  {H}amiltonian dynamics and coisotropic submanifolds.
\newblock {\em Israel J. Math.}, 184:1--57, 2011.

\bibitem[Ush13]{usher-ASENS-2013}
M.~Usher.
\newblock Hofer's metrics and boundary depth.
\newblock {\em Ann. Scient. {\'E}c. Norm. Sup.}, 4(46):57--129, 2013.

\bibitem[UZ16]{usher-zhang-GT-2016}
M.~Usher and J.~Zhang.
\newblock Persistent homology and {F}loer-{N}ovikov theory.
\newblock {\em Geom. Topol.}, 20:3333--3430, 2016.

\bibitem[Vit90]{viterbo-CRASP-1990}
C.~Viterbo.
\newblock Plongements lagrangiens et capacités symplectiques de tores dans
  {$\R^{2n}$}.
\newblock {\em C. R. Acad. Sci. Paris}, 311:487--490, 1990.

\bibitem[Vit92]{viterbo-mathann-1992}
C.~Viterbo.
\newblock Symplectic topology as the geometry of generating functions.
\newblock {\em Math. Ann.}, 292(1):685--710, 1992.

\bibitem[Wei77]{weinstein-book-1977}
A.~Weinstein.
\newblock {\em Lectures on symplectic manifolds}, volume~29 of {\em Regional
  Conference Series in Mathematics}.
\newblock American Mathematical Society, Providence, RI, 1977.

\bibitem[Zap13]{zapolsky-IMRN-2013}
F.~Zapolsky.
\newblock Geometry of contactomorphism groups, contact rigidity, and contact
  dynamics in jet spaces.
\newblock {\em Int. Math. Res. Not.}, 2013(20):4687--4711, 2013.

\bibitem[Zil17]{ziltener-IMRN-2017}
F.~Ziltener.
\newblock Leafwise fixed points for {$C^0$}-small {H}amiltonian flows.
\newblock {\em Int. Math. Res. Not.}, 2019(8):2411--2452, 2017.
\newblock Advance Access Publication.

\end{thebibliography}
\end{document}